\documentclass[a4paper,10pt]{amsart}

\usepackage[utf8]{inputenc}

\usepackage[foot]{amsaddr}
\usepackage{amsmath}
\usepackage{amstext}
\usepackage{amsfonts}
\usepackage{amssymb}
\usepackage{enumitem}

\usepackage{graphicx}
\usepackage{xcolor}
\usepackage{dsfont}
\usepackage{nicefrac}
\usepackage{placeins}

\usepackage{colonequals}

\usepackage{hyperref}
\hypersetup{colorlinks=true, linkcolor=blue, citecolor=blue}

\numberwithin{equation}{section}

\newcommand{\dual}[3][]{\langle #2 , #3\rangle_{#1}}
\newcommand{\dualb}[3][]{\big\langle#2,#3\big\rangle_{#1}}
\newcommand{\dualB}[3][]{\Big\langle#2,#3\Big\rangle_{#1}}
\newcommand{\diff}[1]{\,\mathrm{d}#1}
\newcommand{\R}{\mathbb{R}}
\newcommand{\N}{\mathbb{N}}

\newcommand{\F}{\mathcal{F}}
\renewcommand{\P}{\mathbf{P}}
\newcommand{\E}{\mathbf{E}}

\renewcommand{\exp}[1]{\mathrm{e}^{#1}}
\newcommand{\exptext}[1]{\mathrm{exp}(#1)}
\newcommand{\exptextb}[1]{\mathrm{exp}\big(#1\big)}
\newcommand{\exptextB}[1]{\mathrm{exp}\Big(#1\Big)}

\DeclareMathOperator*{\argmin}{arg\,min}

\theoremstyle{plain}

\newtheorem{definition}{Definition}[section]
\newtheorem{theorem}[definition]{Theorem}
\newtheorem{lemma}[definition]{Lemma}

\newtheorem{assumption}{Assumption}

\theoremstyle{definition}
\newtheorem{remark}[definition]{Remark}

\begin{document}

\thanks{
  This work was partially supported by the Wallenberg AI, Autono\-mous Systems 
  and Software Program (WASP) funded by the Knut and Alice Wallenberg Foundation. 
  The computations were enabled by resources provided by the Swedish National 
  Infrastructure for Computing (SNIC) at LUNARC partially funded by the Swedish 
  Research Council through grant agreement no. 2018–05973.
}

\title[Tamed Euler]{Tamed Euler}
\title[Sub-linear convergence of TSGD in Hilbert space]{Sub-linear convergence of a 
tamed stochastic gradient descent method in Hilbert space}

\author[M.~Eisenmann]{Monika Eisenmann}
\email{monika.eisenmann@math.lth.se}

\author[T.~Stillfjord]{Tony Stillfjord}
\email{tony.stillfjord@math.lth.se}

\address{  Centre for Mathematical Sciences\\
  Lund University\\
  P.O.\ Box 118\\
  221 00 Lund, Sweden}


\begin{abstract}
  In this paper, we introduce the tamed stochastic gradient descent method (TSGD) for
  optimization problems. Inspired by the tamed Euler scheme, which is a 
  commonly used method within the context of stochastic differential equations, TSGD
  is an explicit scheme that exhibits stability properties similar to those of implicit
  schemes. As its computational cost is essentially equivalent to that of the well-known
  stochastic gradient descent method (SGD), it constitutes a very competitive alternative
  to such methods.
    
  We rigorously prove (optimal) sub-linear convergence of the scheme for strongly convex
  objective functions on an abstract Hilbert space. The analysis only requires very mild 
  step size restrictions, which illustrates the good stability properties.
  The analysis is based on a priori estimates more frequently encountered in a time
  integration context than in optimization, and this alternative approach provides a
  different perspective also on the convergence of SGD. Finally, we demonstrate the
  usability of the scheme on a problem arising in a context of supervised learning.
  
\end{abstract}

\maketitle
\section{Introduction}\label{section:introduction}

We consider the gradient flow
\begin{align*}
  w' = - \nabla F(w) , \quad w(0) = w_1,
\end{align*}
on the interval $t \in [0, \infty)$ in order to approximate its steady 
state $w^*$ which satisfies $\nabla F(w^*) = 0$.
We are interested in this problem because for a suitable $F$ its solution solves the 
minimization problem
\begin{equation*}
  w^* = \argmin_{w} F(w).
\end{equation*}
Standard optimization methods may thereby be formulated as time-stepping methods for an evolution equation, which provides an alternative viewpoint on their behaviour and on how to analyze them.

We are mainly interested in the case where $F = \frac{1}{N} \sum_{i=1}^N f_i$ is a sum 
of many 
functions $f_i$ of the same type. This setting occurs in, e.g., supervised learning 
applications, where each $f_i$ corresponds to either a single data point or to a small 
subset (batch) of the data. In order to cover also the infinite data case, we assume more 
generally that
\begin{equation*}
  F(w) = \E_{\xi} \big[ f(\xi, w)\big],
\end{equation*}
where $\xi$ is a random variable and $\E_{\xi}$ denotes the corresponding expectation. Then a realization of $\xi$ corresponds to a specific batch.
In supervised learning applications, the amount of data is frequently very large, and computing the full gradient $\nabla F$ is not feasible. Instead, one typically applies \emph{stochastic} methods where instead of $\nabla F$ the gradient $\nabla f(\xi, \cdot)$ is used, see~\cite{BottouCurtisNocedal.2018} for a general overview.

A popular method is stochastic gradient descent (SGD), given by
\begin{equation*}
  w^{n+1} = w^n - \alpha_n \nabla f(\xi_n, w^n),  \quad w^1 = w_1,
\end{equation*}
where $\{\xi_n\}_{n \in \N}$ denotes a sequence of jointly independent random 
variables and $\{\alpha_n\}_{n \in \N}$ is a sequence of step sizes (learning rates). In 
essence, we apply
the standard gradient descent method but in each step only utilize a randomly chosen 
(small) part of $\nabla F$. More advanced methods such as Adam~\cite{KingmaBa.2014} exist as well, but most are still based on the underlying SGD idea.

Viewed as a time-stepping method, SGD is equivalent to an inexact version of the explicit 
(forward) Euler method and thereby suffers from the same stability issues. In particular, 
there is a severe limit on the step sizes $\alpha_n$, $n \in \N$, where the iterates quickly 
explode in 
size if it is violated. On the other hand, for optimal performance, we want to choose the 
step sizes as large as possible, and thus as close to this limit as possible. Since the limit depends on properties of $F$ that 
are not always known, this is difficult. 

Ideally, one would like to instead use an implicit scheme which is unconditionally stable. 
This would remove the step size restrictions altogether. 
In certain cases, such a method 
can be implemented very efficiently and is then the best choice.
See, 
e.g.~\cite{Bianchi.2016, DavisDrusvyatskiy.2019, ESW.2020, PatrascuNecoara.2017, 
RyuBoyd.2016, ToulisAiroldi.2015, ToulisAiroldi.2016, ToulisRennieAiroldi.2014}
 for analyses of this setting.
In general, however, it means that we have to solve an unfeasibly large system of nonlinear equations 
in each step.

The situation is similar for certain stochastic differential equations (SDEs), where it can be shown that the 
explicit (forward) Euler-Maruyama method diverges with probability one, compare 
\cite{HutzenthalerEtAll.2013}. 
The implicit (backward) Euler-Maruyama scheme is too expensive. In this context, the 
\emph{tamed Euler} scheme provides a fully 
explicit alternative, with better stability properties. 
This scheme was introduced for SDEs in \cite{HutzenthalerEtAll.2012} and has been studied 
further in, e.g., \cite{Sabanis.2013, Sabanis.2016}. 
Very recently, the taming idea has also been extended to a setting similar to ours 
involving stochastic gradient Langevin dynamics~\cite{LovasEtAll.2020}, which generalizes 
the deterministic setting from \cite{BrosseEtAll.2019, SabanisZhang.2019}.

We propose to use a method of this type also in the current context, which we call the tamed stochastic gradient descent (TSGD). It is defined by
\begin{equation*}
  w^{n+1} = w^n - \frac{\alpha_n \nabla f(\xi_n, w^n)}{1 + \alpha_n \|\nabla f(\xi_n, w^n)\|},  
  \quad w^1 = w_1.
\end{equation*}
We note that it is a fully explicit scheme. Further, as the step sizes or the gradients tend to zero, the method tends to the SGD. In fact, it is straightforward to show that TSGD is a second-order perturbation of SGD. However, due to the specific rescaling of the gradient, its stability properties are much better and large step sizes do not cause issues.

The main contribution of this paper is a rigorous error analysis of TSGD in a strongly 
convex setting, which demonstrates that it converges as $\mathcal{O}(\frac{1}{n})$. This 
is the optimal rate which can be expected in this stochastic setting. Notably, we require 
very weak or no bounds on the initial step size, and its size only affects the error 
constants in a mild manner. Another feature of our analysis is that we consider the 
problem in a (possibly) infinite-dimensional Hilbert space, which means that the error 
bounds are applicable not only to optimization of $\R^d$-valued data, but to, e.g.\ 
classification of functions. We also directly prove convergence of $\{w^n\}_{n \in \N}$ 
towards $w^*$ 
rather than of $\{F(w^n)\}_{n \in \N}$ towards $F(w^*)$. While these types of convergence 
are equivalent in the current setting, our approach 
provides better error constants for the first type of convergence than using this 
equivalence together with more standard arguments.

We refer to~\cite{BottouCurtisNocedal.2018} for a general overview of optimization 
methods for our problem setting. This includes a general proof of convergence for 
first-order explicit methods in which many similar methods fit. We note that verifying the 
required assumptions for the method suggested here is non-trivial. Furthermore, applying 
such a general result would not highlight the benefit of the scheme. We also note that our 
analysis is based on a different idea which relies on a priori estimates. The same ideas 
can be applied also to, e.g., SGD, which similarly shows convergence without a strict 
step size restriction. The limitation instead shows up in the error constant, which 
becomes infeasibly large. For the proposed method, the error constant is instead of a 
moderate size. Our analysis thus provides a different viewpoint on the convergence of 
these kinds of methods, which does not rely on prescribed step size limitations.

There are other related methods which might be useful in the given context, such as 
implicit-explicit schemes \cite{Bertsekas.2011, BianchiHachem.2016, PatrascuIrofti.2020, 
RyuYin.2019,SalimBianchiHachem.2019}, where only part of the problem is considered in 
an implicit way, and sum-splitting methods 
\cite{RyuYin.2019, ToulisTranAiroldi.2016, TranToulisAiroldi.2015} where the problem is 
decomposed into many small subproblems and each is considered in an implicit way. 
Both of these approaches rely on there being such easily identifiable splittings, which is 
typically not the case in the general setting. More closely related to our proposed method 
are the stabilized Runge-Kutta schemes proposed in 
\cite{AbdulleMedovikov.2001, vanDerHouwenSommeijer.1980} for parabolic problems 
rather than optimization. See also e.g.~\cite{HairerWanner.2010,Verwer.1996} and 
\cite[Section V]{HundsdorferVerwer.2003} for an overview. Recently, they were adapted to 
solve a special class of deterministic optimization problems in~\cite{EftekhariEtal.2021}. 

While our proofs of convergence require rather strong assumptions, such as strong 
convexity, we hasten to add that the method performs well also in more general settings, 
such as that of general neural networks. This is demonstrated by our numerical 
experiments in Section~\ref{section:experiments}. It is therefore likely that our 
assumptions can be much weakened while still guaranteeing, e.g., local convergence to a 
local minimum. Such considerations would, however, add a considerable amount of 
technical details that would obscure the general idea, and we thus choose to limit 
ourselves to this setting.

The paper is organized as follows. In Section~\ref{section:preliminaries} we fix the 
notation and state the basic assumptions on the optimization problem. Then we formally 
introduce the method in Section~\ref{section:TSGD}. As stated above, our main proof 
relies on a priori estimates, and we prove these in Section~\ref{section:apriori}. These are 
then used in the main error analysis in Section~\ref{section:erroranalysis}. In 
Section~\ref{section:experiments}, we provide several numerical experiments that 
illustrate our claims, both in a setting satisfying our basic assumptions and in a more 
general setting. Section~\ref{section:conclusions} summarises our conclusions. Finally, 
we collect some generally applicable results that are critical for our analysis, but whose 
proofs are overly technical and do not contribute to an understanding of the main ideas in 
Appendix~\ref{section:auxiliary}. 

\section{Preliminaries}\label{section:preliminaries}

In the following, let $(H, \dual{\cdot}{\cdot}, \|\cdot \|)$ be a real Hilbert space. Its 
dual space is denoted
by $(H^*, \dual[H^*]{\cdot}{\cdot}, \| \cdot \|_{H^*} )$. 
Since $H$ is a Hilbert space, there exists an isometric isomorphism $\iota \colon 
H^* \to H$ such that $\iota^{-1} \colon H \to H^*$ with $\iota^{-1}: v \mapsto 
\dual{v}{\cdot}$. 
We denote by $\N$ the natural numbers, not including $0$.

Let $(\Omega, \F, \P )$ be a complete probability space and let $\{\xi_n\}_{n \in \N}$ be 
a family of jointly independent random variables on $\Omega$.
For a random variable $X \colon \Omega \to H$, let $\E_{\xi}[X]$ denote the 
expectation with respect to the probability distribution of $\xi$. 
We are mainly interested in the total expectation
\begin{equation*}
  \E_n \big[ \|X\|^2 \big] = \E_{\xi_1}\big[ \E_{\xi_2}\big[ \cdots 
  \E_{\xi_n} \big[ \|X \|^2 \big] \cdots \big]\big].
\end{equation*}
Since the random variables $\{\xi_n\}_{n \in \N}$ are jointly independent, this expectation 
coincides with the expectation with respect to the joint probability distribution of $\xi_1, 
\ldots, \xi_n$. We also note here that if one of the following statements does not involve an expectation but does contain a random variable, then it is assumed to hold almost surely (a.s.) even if this is not explicitly stated.

For a random variable $\xi$ on $\Omega$, we consider the function $f(\xi, \cdot ) \colon 
\Omega \times H \to \R$ such that
\begin{equation*}
  F(w) = \E_{\xi} \big[ f(\xi, w)\big],
\end{equation*}
and aim to find
\begin{equation*}
  w^* = \argmin_{w} F(w).
\end{equation*}
The  existence of such a minimum will be guaranteed by a strong convexity assumption below. We note that this means that $\nabla F(w^*) = 0$.

Below, we collect all the assumptions that will be used throughout the paper. Each lemma 
and theorem specifies which particular assumptions are in effect at that point. The first 
assumption concerns the 
properties of the functions $f(\xi, \cdot)$, which will be used as stochastic approximations 
to $F$. 

\begin{assumption} \label{ass:fStoch}
  Let $f(\xi, \cdot ) \colon \Omega \times H \to \R$ be given such that
  \begin{itemize}
    \item $\dual{\iota \nabla f (v)}{w} = \lim_{h \to 0} \frac{f(v + hw) - f(v)}{h}$ a.s. 
      for all $v,w \in H$, i.e., $f$ is G\^{a}teaux differentiable a.s.;
    \item there exists 
      $\mu_{\xi} \colon \Omega \to [0,\infty)$ with $\E_{\xi} [\mu_{\xi}] =: \mu \in (0,\infty)$ 
      such that
      \begin{equation*}
        \dual{\iota \nabla f(\xi,v) - \iota \nabla f(\xi,w)}{v - w} \geq \mu_{\xi} \|v - w\|^2 \quad 
        \text{a.s. for all } v,w \in H;
      \end{equation*}
    \item there exists 
    $L_{\xi} \colon \Omega \to [0,\infty)$ with $\big(\E_{\xi}\big[ L_{\xi}^2 
    \big]\big)^{\frac{1}{2}} =: L \in (0,\infty)$ such that
    \begin{equation*}
      \| \iota \nabla f(\xi, v) - \iota \nabla f(\xi, w) \| \leq L_{\xi} \| v - w\| \quad 
      \text{a.s. for all } v,w \in H;
    \end{equation*}
    \item $\big(\E_{\xi} \big[\| \iota \nabla f(\xi,w^*) \|^2 \big]\big)^{\frac{1}{2}} =: 
    \sigma \in [0,\infty)$.
  \end{itemize}
\end{assumption}

The above assumption is enough to prove convergence with a sub-optimal rate and the 
optimal rate in some cases. To guarantee the optimal rate in all cases, we additionlly 
make the following assumption on certain higher moments.

\begin{assumption}\label{ass:fHigherIntegrability}
  Let $f$ be given such that Assumption~\ref{ass:fStoch} is fulfilled. Further, assume that 
  for all $v,w \in H$
  \begin{itemize}
    \item $\dual{\iota \nabla f(\xi,v) - \iota \nabla f(\xi,w)}{v - w} \geq \mu_{\xi} \|v - w\|^2$ 
    with $\big(\E_{\xi} \big[\| \mu_{\xi}^2 \big]\big)^{\frac{1}{2}} =: 
    \mu_2 \in (0,\infty)$.
    \item $ \| \iota \nabla f(\xi, v) - \iota \nabla f(\xi, w) \| \leq L_{\xi} \| v - w\|$ with 
    $\big(\E_{\xi}\big[ L_{\xi}^4 \big]\big)^{\frac{1}{4}} =: L_4 \in (0,\infty)$;
    \item $\big(\E_{\xi} \big[\| \iota \nabla f(\xi,w^*) \|^4 \big]\big)^{\frac{1}{4}} =: 
    \sigma_4 \in [0,\infty)$.
  \end{itemize}
\end{assumption}

Finally, in the case that the gradient is also globally bounded, the convergence result can 
be further improved. For technical reasons we also need to ensure that at points away 
from the minimum, the stochastic gradients are not significantly smaller than they are at 
the minimum of $F$. This is the content of the next assumption.
\begin{assumption}\label{ass:fBounded}
  Let $f$ be given such that Assumption~\ref{ass:fStoch} is fulfilled, and such that
  there exists $B \in (0,\infty)$ with $\| \iota \nabla f(\xi,w)\| \leq B$ a.s.\ for all $w \in H$. 
  Further, there exists $D \in [0,\infty)$ such that 
  \begin{align*}
    \Big(\E_{\xi} \Big[ \chi_{\|\iota \nabla f(\xi,w)\| >0} \frac{\|\iota \nabla f(\xi,w^*)\|^2 }{\|\iota 
    \nabla f(\xi,w)\|^2} \Big]\Big)^{\frac{1}{2}} \leq D
  \end{align*}
  is fulfilled for all $w \in H$.
\end{assumption}

As shown in the auxiliary Lemma~\ref{lem:swap_diff_and_expectation}, 
Assumption~\ref{ass:fStoch} means that $F$ is also G\^{a}teaux differentiable and 
$\nabla F = \E_{\xi}\big[ \nabla f(\xi, \cdot)\big]$. The following lemma summarises a few 
further consequences of the above assumptions.

 \begin{lemma} \label{lem:Fconsequences}
  Let Assumption~\ref{ass:fStoch} be fulfilled. 
  Then $F$ is strongly convex with convexity constant $\mu$ and $\nabla F$ is Lipschitz 
  continuous with Lipschitz constant $L_F \le L$, i.e.\ for all $v,w \in H$ it holds that
  \begin{align*}
    \| \nabla F(v) - \nabla F(w) \|_{H^*} 
    &= \| \iota \nabla F(v) - \iota \nabla F(w) \| 
    \leq L_F \| v - w\| 
    \leq L \| v - w\| \quad \text{and} \\
    F(v) &\geq F(w) + \dual{\iota \nabla F(w) }{ v - w }+ \frac{\mu}{2} \|v - w\|^2.
  \end{align*}
  Further, the first inequality implies that
  \begin{align*}
    F(v) 
    &\leq F(w)+\dual{\iota \nabla F(w) }{ v - w }+ \frac{L}{2} \|v - w\|^2.
  \end{align*}
  Finally, there exists a unique $w^* \in H$ such that $F(w^*) = \min_{w \in H} F(w)$.
\end{lemma}

\begin{proof}
  Using $\nabla F(w) = \E_{\xi} \big[ \nabla f(\xi, w)\big]$, we obtain
  \begin{align*}
    \dual{\iota \nabla F(v) - \iota \nabla F(w)}{v - w} \geq \mu \|v - w\|^2 \quad 
    \text{for all } v,w \in H.
  \end{align*}
  Thus, the function $v\mapsto \nabla F(v) - \mu \iota^{-1} v$ is monotone. Applying 
  \cite[Proposition~25.10]{ZeidlerIIB.1990}, it follows that $v \mapsto F(v) - \frac{\mu}{2} 
  \|v\|^2$ is convex such that $F$ is strongly convex and the variational inequality stated 
  in the lemma is fulfilled. The Lipschitz continuity of $\iota \nabla F$ 
  similarly follows from the Lipschitz continuity of $\iota \nabla f(\xi, \cdot)$ by the 
  identification provided in Lemma~\ref{lem:swap_diff_and_expectation}. The final 
  inequality follows by expanding $F$ in a zeroth-order Taylor expansion 
  around $w$ and using the Lipschitz continuity. See e.g.\ 
  \cite[Appendix~B]{BottouCurtisNocedal.2018} for more details.
  We note that $F$ is coercive since it is strongly convex. Combined with the G\^{a}teaux 
  differentiability, this guarantees the existence of a unique global 
  minimum, see e.g.~\cite[Theorem 25.D, Proposition~25.20 and 
  Corollary~25.15]{ZeidlerIIB.1990}.
\end{proof}

\section{The stochastic tamed Euler scheme} \label{section:TSGD}
Throughout the paper, we will assume that $\{\xi_n\}_{n \in \N}$ is a given a family of 
jointly independent random variables and we will abbreviate $f_n(w) = f(\xi_n,w)$ for $n 
\in \N$. The $\xi_n$ typically correspond to what batches we choose in each iteration, 
i.e.\ on which part of the data we evaluate the gradient.
Let $\{\alpha_n\}_{n \in \N}$ be a sequence of of positive real numbers.
We then consider the stochastic tamed Euler scheme
\begin{align}\label{eq:stochTamedEuler}
  w^{n+1} = w^n - \frac{\alpha_n \iota \nabla f_n(w^n)}{1 + \alpha_n \|\iota 
  \nabla f_n(w^n)\|} \quad \text{for }n \in \N, 
  \quad w^1 = w_1.
\end{align}
Note that it is also possible to choose a random initial value $w^1$, and our convergence statements can be 
extended to this setting in a straightforward way. For simplicity, we restrict ourselves to a fixed initial 
value $w_1 \in H$ in the following.

We note that the computational effort of the scheme is essentially the same as that of 
SGD, since once $\nabla f_n(w^n)$ has been found it is cheap to compute its norm.
We also note that TSGD can be interpreted as a second order perturbation of SGD, since
\begin{equation*}
  \frac{\alpha_n \iota \nabla f_n(w^n)}{1 + \alpha_n \|\iota 
  \nabla f_n(w^n)\|} = \alpha_n \iota \nabla f_n(w^n) - \frac{\alpha_n^2 \|\nabla 
  f_n(w^n)\|\iota \nabla f_n(w^n)}{1 + \alpha_n \|\iota \nabla f_n(w^n)\|}.
\end{equation*}
This second order perturbation mainly offers advantages if $\alpha_n \|\iota \nabla 
f_n(w^n)\|$ is large. In this case we make use of the fact that
\begin{align*}
  \frac{1}{2}\min\big\{1, \alpha_n \|\iota \nabla f_n(w^n)\|\big\}
  \leq \frac{\alpha_n \|\iota \nabla f_n(w^n)\|}{1 + \alpha_n \|\iota \nabla f_n(w^n)\|}
  \leq \min\big\{1, \alpha_n \|\iota \nabla f_n(w^n)\|\big\}.
\end{align*}
Thus, the growth of $w^{n+1} - w^n = \frac{-\alpha_n \iota \nabla f_n(w^n)}{1 + \alpha_n 
\|\iota \nabla f_n(w^n)\|}$ is always bounded.

\section{A priori bounds}\label{section:apriori}

Our main results will show that $\E_n \big[\|w^{n+1} - w^*\|^2\big]$ tends to zero as 
$\frac{1}{n}$ in the strongly convex case. In this section, we prepare for the proofs of 
this by first showing that the errors are bounded.
We note that the main argument here only requires convexity rather than strong convexity, and the $w^*$ in the following three lemmas could therefore equally well be any $w^* \in H$ that satifies $\nabla F(w^*) = 0$.

\begin{lemma} \label{lem:apriori2}
  Let Assumption~\ref{ass:fStoch} be fulfilled and let $\{\alpha_n\}_{n \in \N}$ be a 
  sequence of positive real numbers such that $\sum_{n =1}^{\infty} \alpha_n^2 < 
  \infty$.
  For $\Phi \in [0,\infty)$ the a priori bound
  \begin{align*}
    &\E_n \big[\|w^{n+1} - w^*\|^2\big] \\
    &\leq \| w_1 - w^*\|^2 \exptextB{ \sum_{i=1}^{\infty} \big(2 \sigma^2 \min\big\{\Phi^{-2}, 
      \alpha_i^2 \big\} + 4 \alpha_i^2 L^2 m_i\big)} \\
    &\quad + \sum_{i=1}^{\infty} \big(\Phi^2 \min \big\{ \E_i \big[ \|\iota \nabla 
    f_i(w^i)\|^{-2} \big], 
    \alpha_i^2 \big\} + 2 ( 1- m_i) + 2 \min \big\{1, 2 \alpha_i^2 \sigma^2 \big\}\big)\\
    &\qquad \times \exptextB{\sum_{j=i+1}^{\infty} \big(2 \sigma^2 \min\big\{\Phi^{-2}, 
      \alpha_j^2 \big\} + 4 \alpha_j^2 L^2 m_j\big)} =: M_2
  \end{align*}
  is fulfilled, where $\min\big\{\Phi^{-2}, x \big\} = x$ for $\Phi = 0$ and every $x \in \R$ 
  and 
  \begin{align*}
    m_i =
    \begin{cases}
      1, & \text{\emph{if }} 2 \alpha_i^2 L^2 \E_{i-1} \big[ \| w^i - w^*\|^2\big] \le 1, \\
      0, & \text{\emph{otherwise}}.
    \end{cases}
  \end{align*}
  Furthermore, there exists $n_0 \in \N$ such that $m_i = 1$ for all $i \geq n_0$.
\end{lemma}

\begin{remark}
  The advantage of this particular a priori bound is that the bound does not grow very 
  much when the initial step size is increased. The corresponding proof for the 
  SGD method looks very similar, but does not have the factors 
  $m_n$ or $\min\{\dots, \alpha_n^2\}$, $n \in \N$. This means that the first few terms 
  in the products become very large, even for moderately sized Lipschitz constants, 
  reflecting the fact that a too large step size can 
  lead to instability. In our case, these large terms are multiplied by $0$ or cut off by the 
  $\min$-function.
  The constant $\Phi$ can be used to tune the error bound further in case $\sigma$ or 
  $\E_n \big[ \|\iota \nabla f_n( w^n)\|^{-2} \big]$, $n \in \N$, is large.
\end{remark}

\begin{proof}[Proof of Lemma~\ref{lem:apriori2}]
  We test the scheme defined by~\eqref{eq:stochTamedEuler}
  with $w^n - w^*$, in order to obtain that
  \begin{align}\label{eq:proofApriori1}
    \begin{split}
      &\dual{w^{n+1} - w^* - (w^n - w^*)}{w^n - w^*} \\
      &\quad + \frac{\alpha_n \dual{\iota \nabla f_n(w^n) - \iota \nabla 
          f_n(w^*)}{w^n - w^*}}{1 
        + \alpha_n \|\iota \nabla f_n(w^n)\|} 
      = - \frac{\alpha_n \dual{\iota \nabla f_n(w^*)}{w^n - w^*}}{1 + \alpha_n \|\iota 
        \nabla f_n(w^n)\|}.
    \end{split}
  \end{align}
  Using the identity $\dual{u-v}{u} = \frac{1}{2} (\|u\|^2 - \|v\|^2 + \|u-v\|^2)$, $u,v \in H$, the 
  first summand on the left-hand side can be written as
  \begin{align*}
    &- \dual{w^n - w^* - (w^{n+1} - w^*)}{w^n - w^*} \\
    &= - \frac{1}{2} \big( \|w^n - w^*\|^2 - \|w^{n+1} - w^*\|^2 + \|w^{n+1} - w^n\|^2 \big)\\
    &= \frac{1}{2} \Big(\|w^{n+1} - w^*\|^2 - \|w^n - w^*\|^2 - \frac{\alpha_n^2 \| \iota 
      \nabla  f_n(w^n)\|^2 }{(1 + \alpha_n \|\iota \nabla f_n(w^n)\|)^2}\Big), 
  \end{align*}
  where we inserted the scheme in the last step. 
  Thus, inserting the monotonicity condition for $f_n$ into 
  \eqref{eq:proofApriori1} and multiplying the inequality with the factor two, it follows that
  \begin{align*}
    &\|w^{n+1} - w^*\|^2 - \|w^n - w^*\|^2 \\
    &\leq - \frac{2\alpha_n \dual{\iota \nabla f_n(w^*)}{w^n - w^*}}{1 + \alpha_n 
      \|\iota \nabla  
      f_n(w^n)\|} 
    + \frac{\alpha_n^2 \| \iota \nabla f_n(w^n)\|^2 }{(1 + \alpha_n \|\iota \nabla  
      f_n(w^n)\|)^2} 
    =: I_1 + I_2.
  \end{align*}
  Since the tamed Euler scheme is the forward Euler scheme with a second order 
  perturbation, it follows that
  \begin{align*}
    \frac{\alpha_n \iota \nabla f_n(w^*)}{1 + \alpha_n \|\iota \nabla  f_n(w^n)\|}
    = \frac{\alpha_n \iota \nabla f_n(w^*)}{1 + \alpha_n \Phi}
    + \frac{ \alpha_n^2 \big(\Phi - \|\iota \nabla f_n(w^n)\|\big) \iota \nabla f_n(w^*) 
    }{(1 + \alpha_n \|\iota \nabla  f_n(w^n)\|) (1 + \alpha_n \Phi )}
  \end{align*}
  for $\Phi \in [0,\infty)$. Note that we have $w^*$ in the numerator of the left-hand-side 
  but $w^n$ in the denominator. We insert this equality into $I_1$ and use the 
  Cauchy--Schwarz inequality and 
  Young's  inequality for products in order to obtain
  \begin{align*}
    I_1 
    &= - \frac{\alpha_n \dual{2 \iota \nabla f_n(w^*)}{w^n - w^*}} {1 + \alpha_n \|\iota \nabla  
      f_n(w^n)\|} \\
    &\leq - \frac{2 \alpha_n \dual{\iota \nabla f_n(w^*)}{w^n - w^*} }{1 + \alpha_n \Phi}
    + \frac{2 \alpha_n^2 \Phi \| \iota \nabla f_n(w^*)\| \| w^n - w^*\| }{(1 + \alpha_n \|\iota 
    \nabla  f_n(w^n)\|) (1 + \alpha_n \Phi )} \\
    &\quad + \frac{2 \alpha_n^2 \|\iota \nabla f_n(w^n)\| \| \iota \nabla f_n(w^*)\| \| w^n - 
    w^*\| }{(1 + \alpha_n \|\iota \nabla  f_n(w^n)\|) (1 + \alpha_n \Phi )} \\
    &\leq - \frac{2 \alpha_n \dual{\iota \nabla f_n(w^*)}{w^n - w^*} }{1 + \alpha_n \Phi}
    + \frac{ \alpha_n^2 \big(\Phi^2 + \|\iota \nabla f_n(w^n)\|^2\big)}{(1 + \alpha_n \|\iota 
    \nabla  f_n(w^n)\|)^2} \\
    &\quad + \frac{2 \alpha_n^2 \| \iota \nabla f_n(w^*)\|^2 \| w^n - w^*\|^2 }{(1 + \alpha_n 
    \Phi )^2}
    =: I_{1,1} + I_{1,2} + I_{1,3}.
  \end{align*}
  For $I_{1,1} = - \frac{2 \alpha_n \dual{\iota \nabla f_n(w^*)}{w^n - w^*} }{1 + \alpha_n 
  \Phi}$, we notice that $\E_{\xi_n} [I_{1,1}] = 0$. Moreover, for $I_{1,2} = \frac{\alpha_n^2 
  ( 
  \Phi^2 + \|\iota \nabla f_n(w^n)\|^2 )}{(1 + \alpha_n \|\iota \nabla f_n(w^n)\|)^2}$, we get
  \begin{align*}
    I_{1,2}
    &\leq \Phi^2 \min \big\{ \|\iota \nabla f_n(w^n)\|^{-2} , \alpha_n^2 \big\}
    + \min \big\{1, \alpha_n^2 \|\iota \nabla f_n(w^n)\|^2 \big\} \\
    &\leq \Phi^2 \min \big\{ \|\iota \nabla f_n(w^n)\|^{-2} , \alpha_n^2 \big\}
    + \min \big\{1, 2 \alpha_n^2 L_{\xi_n}^2 \|w^n - w^*\|^2\big\}\\ 
    &\quad + \min \big\{1,2 \alpha_n^2 \|\iota \nabla f_n(w^*)\|^2 \big\}
  \end{align*}
  and
  \begin{align*}
    I_{1,3}
    = \frac{2\alpha_n^2 \|\iota \nabla f_n(w^*)\|^2 \|w^n - w^*\|^2}{(1 + \alpha_n \Phi )^2}
    \leq 2 \|\iota \nabla f_n(w^*)\|^2 \min \big\{\Phi^{-2}, \alpha_n^2 \big\} \|w^n - w^*\|^2.
  \end{align*}
  A bound for $I_2$ is given by
  \begin{align*}
    I_2 &= \frac{\alpha_n^2 \| \iota \nabla f_n(w^n)\|^2 }{(1 + \alpha_n \|\iota \nabla 
    f_n(w^n)\| )^2}
    \leq \min \big\{1, \alpha_n^2 \| \iota \nabla f_n(w^n)\|^2 \big\}\\
    &\leq \min \big\{1, 2 \alpha_n^2 L_{\xi_n}^2 \| w^n - w^*\|^2 \big\}
    + \min \big\{1, 2 \alpha_n^2 \|\iota \nabla f_n(w^*)\|^2 \big\}.
  \end{align*}  
  Then it follows
  \begin{align}
    \label{eq:aprioriBoundMn}
    \begin{split}
      &\|w^{n+1} - w^*\|^2 - \|w^n - w^*\|^2 \leq  I_1 + I_2\\
      &\leq - \frac{2 \alpha_n \dual{\iota \nabla f_n(w^*)}{w^n - w^*} }{1 + \alpha_n \Phi}
      + \Phi^2 \min \big\{ \|\iota \nabla f_n(w^n)\|^{-2} , \alpha_n^2 \big\}\\
      &\quad + 2 \min \big\{1, 2 \alpha_n^2 L_{\xi_n}^2 \|w^n - w^*\|^2\big\} + 2 \min 
      \big\{1,2 \alpha_n^2 \|\iota \nabla f_n(w^*)\|^2 \big\}\\
      &\quad + 2 \|\iota \nabla f_n(w^*)\|^2 \min \big\{\Phi^{-2}, \alpha_n^2 \big\} \|w^n - 
      w^*\|^2.
    \end{split}
  \end{align}
  Taking the $\E_n$-expectation, we then obtain
  \begin{equation}
    \label{eq:aprioriBoundMnExp}
    \begin{aligned}
      &\E_n \big[\|w^{n+1} - w^*\|^2\big] \\
      &\quad \leq \big(1 + 2 \sigma^2 \min\big\{ \Phi^{-2}, \alpha_n^2 \big\} \big) \E_{n-1} 
      \big[ 
      \|w^n - w^*\|^2 \big]\\
      &\qquad + \Phi^2 \min \big\{ \E_n \big[\|\iota \nabla f_n(w^n)\|^{-2}\big], \alpha_n^2 
      \big\}\\
      &\qquad + 2 \min \big\{1, 2 \alpha_n^2 L^2 \E_{n-1} \big[ \|w^n - w^*\|^2 \big]\big\} 
      + 2 \min \big\{1,2 \alpha_n^2 \sigma^2\big\}\\
      &\quad= \big(1 + 2 \min\big\{ \Phi^{-2}, \alpha_n^2 \big\} \sigma^2+ 4 \alpha_n^2 L^2 
      m_n\big) \E_{n-1} \big[ \|w^n - w^*\|^2 \big]\\
      &\qquad + \Phi^2 \min \big\{ \E_n \big[\|\iota \nabla f_n(w^n)\|^{-2}\big], 
      \alpha_n^2 \big\}
      + 2 (1-m_n) 
      + 2 \min \big\{1,2 \alpha_n^2 \sigma^2\big\},
    \end{aligned}
  \end{equation}
  with $m_n$ defined as in the lemma statement. Reinserting the bound 
  repeatedly thus yields
  \begin{align*}
    &\E_n \big[\|w^{n+1} - w^*\|^2\big] \\
    &\leq \| w_1 - w^*\|^2 \prod_{i=1}^{n} \big(1 + 2 \sigma^2 \min\big\{\Phi^{-2}, 
    \alpha_i^2 \big\} + 4 \alpha_i^2 L^2 m_i\big) \\
    &\quad + \sum_{i=1}^{n} \big(\Phi^2 \min \big\{ \E_i \big[ \|\iota \nabla 
    f_i(w^i)\|^{-2} \big], 
    \alpha_i^2 \big\} + 2 ( 1- m_i) + 2 \min \big\{1, 2 \alpha_i^2 \sigma^2 \big\}\big)\\
    &\qquad \times \prod_{j=i+1}^{n} \big(1 + 2 \sigma^2 \min\big\{\Phi^{-2}, 
    \alpha_j^2 \big\} + 4 \alpha_j^2 L^2 m_j\big).
  \end{align*}
  Finally, we apply the inequality $1 + x \leq \exptext{x}$, $x \in \R$, and make 
  the bound independent of $n$ by bounding the final sums by the corresponding infinite 
  sums, in order to obtain
  \begin{align*}
    &\E_n \big[\|w^{n+1} - w^*\|^2\big] \\
    &\leq \| w_1 - w^*\|^2 \exptextB{ \sum_{i=1}^{\infty} \big(2 \sigma^2 \min\big\{\Phi^{-2}, 
      \alpha_i^2 \big\} + 4 \alpha_i^2 L^2 m_i\big)} \\
    &\quad + \sum_{i=1}^{\infty} \big(\Phi^2 \min \big\{ \E_i \big[ \|\iota \nabla 
    f_i(w^i)\|^{-2} \big], 
    \alpha_i^2 \big\} + 2 ( 1- m_i) + 2 \min \big\{1, 2 \alpha_i^2 \sigma^2 \big\}\big)\\
    &\qquad \times \exptextB{\sum_{j=i+1}^{\infty} \big(2 \sigma^2 \min\big\{\Phi^{-2}, 
      \alpha_j^2 \big\} + 4 \alpha_j^2 L^2 m_j\big)}.
  \end{align*}
  It remains to verify, that there exists $n_0 \in \N$ such that $m_n = 1$ for all $n \geq 
  n_0$. This can be done by estimating \eqref{eq:aprioriBoundMnExp} and following a 
  similar line of argumentation as before. First, we can write
  \begin{align*}
    \E_n \big[\|w^{n+1} - w^*\|^2\big] 
    \leq \big(1 + 2 \alpha_n^2 \sigma^2 + 4 \alpha_n^2 L^2 \big) \E_{n-1} \big[ \|w^n - 
    w^*\|^2\big] + \Phi^2 \alpha_n^2 + 4 \alpha_n^2 \sigma^2.
  \end{align*}
  Reinserting the inequality $n -1$ times, it follows that
  \begin{align*}
    &\E_n \big[\|w^{n+1} - w^*\|^2\big] \\
    &\leq \| w_1 - w^*\|^2 \prod_{i=1}^{n} \big(1 + 2 \sigma^2 \alpha_i^2 + 4 \alpha_i^2 L^2 
    \big) 
    + \sum_{i=1}^{n} 4 \alpha_i^2 \sigma^2 \prod_{j=i+1}^{n} \big(1 + 2 \sigma^2 
    \alpha_j^2 + 4 \alpha_j^2 L^2 \big)\\
    &\leq \| w_1 - w^*\|^2 \exptextB{ \big(2 \sigma^2 + 4 L^2 \big) \sum_{i=1}^{\infty} 
    \alpha_i^2 }
    + \sum_{i=1}^{\infty} 4 \alpha_i^2 \sigma^2 \exptextB{ \big(2 \sigma^2 + 4 L^2 \big) 
    \sum_{j=i+1}^{\infty} \alpha_j^2 }.
  \end{align*}
  Since $\sum_{n =1}^{\infty} \alpha_n^2 < \infty$ there is thus a $n_0 \in \N$ such that $2 \alpha_n^2 L^2 \E_{n-1} \big[ \| w^n - w^*\|^2\big] \leq 1$ for all $n \geq n_0$.
\end{proof}

\begin{lemma} \label{lem:apriori4}
  Let Assumption~\ref{ass:fHigherIntegrability} be fulfilled and 
  let $\{\alpha_n\}_{n \in \N}$ be a sequence of positive real numbers such that $\sum_{n 
  =1}^{\infty} \alpha_n^2 < \infty$.
  Then the a priori bound  
  \begin{align*}
    \E_n \big[\|w^{n+1} - w^*\|^4\big] 
    &\leq \|w^1 - w^*\|^4 \exptextB{\sum_{i = 1}^{\infty} c_1^i(\alpha_i)} \\
    &\quad + \sum_{i=1}^{\infty} \big(c_2^i(\alpha_i) M_2 + c_3^i(\alpha_i)\big) 
    \exptextB{\sum_{j = i+ 1}^{\infty} c_1^j(\alpha_j) } =: M_4
  \end{align*}
  is fulfilled for $c_1^i, c_2^i, c_3^i \colon (0,\infty) \to (0,\infty)$ such that there exist 
  $C_1^k, C_2^k, C_3^k, C_4^k \in (0,\infty)$ with $c_k^i(\alpha) \leq C_1^i \min \{C_2^i, 
  \alpha^4\} + C_3^i \min \{C_4^i, \alpha^2\}$ for all $\alpha \in (0,\infty)$, $k \in \{1,2,3\}$ 
  and $i \in \N$.
\end{lemma}

\begin{proof}
  Within the proof of Lemma~\ref{lem:apriori2}, we verified the inequality 
  \eqref{eq:aprioriBoundMn}. Starting from this point, we find
  \begin{equation}
    \label{eq:apriori4}
    \|w^{n+1} - w^*\|^2 - \|w^n - w^*\|^2 \leq A_n + B_n,
  \end{equation}
  where
  \begin{align*}
    A_n &= -\frac{2 \alpha_n \dual{\iota \nabla f_n(w^*)}{w^n - 
    w^*} }{1 + \alpha_n \Phi} \qquad \text{and} \\
    B_n &= \Phi^2 \min \big\{ \|\iota \nabla f_n(w^n)\|^{-2} , \alpha_n^2 \big\}
     + 2 \min \big\{1, 2 \alpha_n^2 L_{\xi_n}^2 \|w^n - w^*\|^2\big\}\\
      &\quad + 2 \min \big\{1,2 \alpha_n^2 \|\iota \nabla f_n(w^*)\|^2 \big\}
      + 2 \|\iota \nabla f_n(w^*)\|^2 \min \big\{\Phi^{-2}, \alpha_n^2 \big\} \|w^n - w^*\|^2 
  \end{align*}
  for $\Phi \in [0,\infty)$. We note that the parameter $\Phi$ can be chosen such that 
  $\E_n[B_n]$ is as small as possible.
  From this it follows that 
  \begin{align*}
    \E_{\xi_n} [A_n] &= 0,\\
    \E_{\xi_n} [B_n] &\leq  \Phi^2 \min \big\{ \E_{\xi_n} \big[ \|\iota \nabla f_n(w^n)\|^{-2}\big] 
    , 
    \alpha_n^2 \big\}
    + 2 \min \big\{1, 2 \alpha_n^2 L^2 \|w^n - w^*\|^2\big\}\\
    &\quad + 2 \min \big\{1,2 \alpha_n^2 \sigma^2 \big\}
    + 2 \sigma^2 \min \big\{\Phi^{-2}, \alpha_n^2 \big\} \|w^n - w^*\|^2, \\
    \E_{\xi_n} [A_n^2] 
    &\leq 4 \sigma^2 \min\{ \Phi^{-2}, \alpha_n^2 \} \| w^n - w^*\|^2 \qquad \text{and}\\
    \E_{\xi_n} [B_n^2] 
    &\leq 4 \Phi^4 \min \big\{ \big(\E_{\xi_n} \big[ \|\iota \nabla f_n(w^n)\|^{-2}\big]\big)^2 , 
    \alpha_n^4 \big\}
    + 16 \min \big\{1, 4 \alpha_n^4 L_4^4 \|w^n - w^*\|^4\big\}\\
    &\quad + 16 \min \big\{1,4 \alpha_n^4 \sigma_4^4 \big\}
    + 16  \sigma_4^4 \min \big\{\Phi^{-4}, \alpha_n^4 \big\} \|w^n - w^*\|^4.
  \end{align*}
   We note that for $a,b \in \R$ we have the identity $(a - b)a = \frac{1}{2} \big( |a|^2 - |b|^2 
   + |a - b|^2 \big)$, and thus $|a|^2 - |b|^2 \le 2(a-b)a$. By multiplying the inequality from 
   \eqref{eq:apriori4} with the factor $2\|w^{n+1} - w^*\|^2$, we therefore obtain
  \begin{align*}
    \|w^{n+1} - w^*\|^4 - \|w^n - w^*\|^4
    &\leq 2\big( A_n + B_n \big) \|w^{n+1} - w^*\|^2\\
    &\leq 2\big( A_n + B_n \big) \big( \|w^n - w^*\|^2 + A_n + B_n  \big)\\
    &\leq 2(A_n + B_n) \|w^n - w^*\|^2  + 4 A_n^2 + 4 B_n^2 
  \end{align*}
  and in $\E_{\xi_n}$-expectation
  \begin{align*}
    & \E_{\xi_n} \big[ \|w^{n+1} - w^*\|^4\big] - \|w^n - w^*\|^4\\
    &\leq 4 \Big(2 L^2 \min \big\{\tfrac{1}{2} L^{-2} \|w^n - w^*\|^{-2}, \alpha_n^2 \big\}
    + \sigma^2 \min \big\{\Phi^{-2}, \alpha_n^2 \big\}\\
    &\qquad + 64 L_4^4 \min \big\{ \tfrac{1}{4} L_4^{-4} \|w^n - w^*\|^{-4}, \alpha_n^4 
    \big\} + 16 \sigma_4^4 \min \big\{\Phi^{-4}, \alpha_n^4 \big\} \Big)\|w^n - w^*\|^4 \\
    &\quad + 2 \Big( \Phi^2 \min \big\{ \E_{\xi_n} \big[ \|\iota \nabla f_n(w^n)\|^{-2}\big] , 
    \alpha_n^2 \big\} 
    + 4 \sigma^2\min \big\{\tfrac{1}{2} \sigma^{-2} , \alpha_n^2 \big\} \\
    &\qquad + 8  \sigma^2 \min\{ \Phi^{-2}, \alpha_n^2 \} \Big) \|w^n - w^*\|^2 \\
    &\quad + 16 \Phi^4 \min \big\{ \big(\E_{\xi_n} \big[ \|\iota \nabla 
    f_n(w^n)\|^{-2}\big]\big)^2 
    , \alpha_n^4 \big\}
    + 256 \sigma_4^4 \min \big\{ \tfrac{1}{4}\sigma_4^{-4},\alpha_n^4 \big\}\\
    &=: c_1^n (\alpha_n) \|w^n - w^*\|^4 + c_2^n(\alpha_n) \|w^n - w^*\|^2 
    + c_3^n(\alpha_n).
  \end{align*}
  Adding $\|w^n - w^*\|^4$ to both sides of the inequality, taking the 
  $\E_{n-1}$-expectation and reinserting the bound, we obtain
  \begin{align*}
    & \E_n \big[ \|w^{n+1} - w^*\|^4\big] \\
    &\leq \big(1 + c_1^n(\alpha_n) \big) \E_{n-1} \big[ \|w^n - w^*\|^4 \big]
    + c_2^n(\alpha_n) M_2 + c_3^n(\alpha_n)\\
    &\leq \|w^1 - w^*\|^4 \prod_{i = 1}^{n} \big(1 + c_1^i(\alpha_i) \big) 
    + \sum_{i=1}^{n} \big(c_2^i(\alpha_i) M_2 + c_3^i(\alpha_i)\big) \prod_{j = i+ 
    1}^{n} \big(1 + c_1^j(\alpha_j) \big) \\
    &\leq \|w^1 - w^*\|^4 \exptextB{\sum_{i = 1}^{\infty} c_1^i(\alpha_i) } 
    + \sum_{i=1}^{\infty} \big(c_2^i(\alpha_i) M_2 + c_3^i(\alpha_i)\big) 
    \exptextB{\sum_{j = i+ 1}^{\infty} c_1^j(\alpha_j)} .
  \end{align*}
  Finally, this is finite due to the assumption $\sum_{n  =1}^{\infty} \alpha_n^2 < \infty$.
\end{proof}

It is much easier to show the following \emph{pathwise} a priori bound, which provides the 
intuition for why the scheme is good; in $n$ steps, we can only make the error worse by 
$n$ in the worst case. This is a marked improvement over the situation for other explicit 
methods such as SGD, where the error may grow without bound. It is in fact similar to 
what one would get from an implicit scheme such as the implicit Euler, corresponding to 
the proximal point method in the context of optimization.

\begin{lemma} \label{lem:aprioriPath}
  Let $f(\xi, \cdot ) \colon \Omega \times H \to \R$ be G\^{a}teaux differentiable a.s.\ and
  let $\{\alpha_n\}_{n \in \N}$ be a sequence of positive real numbers.
  Then the a priori bound
  \begin{align*}
    \|w^{n+1} - w^*\|
    \leq \|w_1 - w^*\| + \sum_{i=1}^{n} \min \{1, \alpha_n \|\iota \nabla f_i(w^i)\| \}
    \leq \|w_1 - w^*\| + n
  \end{align*}
  is fulfilled.
\end{lemma}

\begin{proof}
  We recall the TSGD scheme from \eqref{eq:stochTamedEuler}
  and obtain that
  \begin{align*}
    \|w^{n+1} - w^*\|
    &\leq \|w^{n+1} - w^n\| + \|w^n - w^*\|\\
    &= \frac{\alpha_n \|\iota \nabla f_n(w^n)\| }{1 + \alpha_n \|\iota \nabla 
    f_n(w^n)\|}
    + \|w^n - w^*\|\\
    &\leq \min \{1, \alpha_n \|\iota \nabla f_n(w^n)\| \} + \|w^n - w^*\|.
  \end{align*}
  Reinserting this inequality shows that $\|w^{n+1} - w^*\| \leq \|w_1 - w^*\| + n$ holds.
\end{proof}

\section{Error analysis}\label{section:erroranalysis}
Given  $z \in H$ and $\alpha > 0$, we define $T_{\alpha f_n, z}(w) \colon \Omega \times 
H \to H$ by
\begin{align*}
  T_{\alpha f_n, z}(w)
  = w - \frac{\alpha \iota \nabla f_n(w)}{1 + \alpha \|\iota \nabla f_n(z)\|}.
\end{align*}
This implies that the next iterate $w^{n+1}$ is given by $T_{\alpha_n f_n, w^n}(w^n)$.

\begin{lemma}\label{lem:SecondLipT}
  Let Assumption~\ref{ass:fStoch} be fulfilled and let $z \in H$, $\alpha 
  \in (0,\infty)$ and $\Xi \in [0,\infty)$ be given.
  It then follows that
  \begin{align*}
    \Big\| T_{\alpha f_n, z}(w) - w + \frac{\alpha 
      \iota \nabla f_n(w)}{1 + \alpha \Xi} \Big\|
    = \frac{ \alpha^2 \big| \Xi - \|\iota \nabla f_n(z)\|\big| \|\iota \nabla f_n(w)\|}
    {(1 + \alpha \|\iota \nabla f_n(z)\| ) (1 + \alpha \Xi )}
  \end{align*}
  for all $w \in H$.
\end{lemma}

\begin{proof}
  Inserting the definition of $T_{\alpha f_n, z}$, it follows that
  \begin{align*}
    &\Big\| T_{\alpha f_n, z}(w) - w + \frac{\alpha \iota \nabla f_n(w)}{1 + \alpha 
    \Xi} \Big\|
  = \Big\| \frac{\alpha \iota \nabla f_n(w)}{1 + \alpha \|\iota \nabla f_n(z)\|} - 
    \frac{\alpha \iota \nabla f_n(w)}{1 + \alpha \Xi } \Big\|\\
    &= \Big\| \frac{\alpha \iota \nabla f_n(w) + \alpha^2 \Xi \iota \nabla f_n(w) 
    }{(1 + \alpha \|\iota \nabla f_n(z)\| ) (1 + \alpha \Xi )}
    - \frac{\alpha \iota \nabla f_n(w) + \alpha^2 \|\iota \nabla f_n(z)\| \iota \nabla 
    f_n(w) }{ (1 + \alpha \|\iota \nabla f_n(z)\| ) (1 + \alpha \Xi ) } \Big\|\\
    &= \frac{ \alpha^2 \big| \Xi - \|\iota \nabla f_n(z)\|\big| \|\iota \nabla f_n(w)\|}
    {(1 + \alpha \|\iota \nabla f_n(z)\| ) (1 + \alpha \Xi )},
  \end{align*}
  for all $w \in H$, which proves the claim.
\end{proof}

\begin{lemma} \label{lem:boundForTheo}
  Let Assumption~\ref{ass:fStoch} be fulfilled and let 
  $\{\alpha_n\}_{n 
  \in \N}$  be a sequence of positive real numbers.
   For any $\Xi \in [0,\infty)$ it then follows that
  \begin{align*}
    \E_{\xi_n} \big[\| w^{n+1} - w^* \|^2\big]
    &\leq \Big( 1 - \E_{\xi_n} \Big[ \frac{2 \alpha_n \mu_{\xi_n} }{1 + \alpha_n \|\iota \nabla 
    f_n(w^n)\| } \Big] \Big) \| w^n - w^*\|^2\\
    &\quad + 2 \E_{\xi_n} \Big[ \frac{\alpha_n^2 \|\iota \nabla f_n(w^n) - \iota \nabla 
    f_n(w^*)\|^2}{(1 + \alpha_n \|\iota \nabla 
      f_n(w^n)\| )^2} \Big]\\
    &\quad + 2 \E_{\xi_n} \Big[ \frac{\alpha_n^2 \|\iota \nabla f_n (w^*)\|^2 }{( 1 
    + \alpha_n \|\iota \nabla f_n(w^n)\| )^2} \Big]\\
    &\quad + 2\E_{\xi_n} \Big[ \frac{ \alpha_n^2 \big| \Xi - \|\iota  \nabla f_n(w^n)\| \big| 
    \|\iota \nabla f_n(w^*)\|}
    {(1 + \alpha_n \|\iota \nabla f_n(w^n)\| ) (1 + \alpha_n \Xi )} \Big] 
    \| w^n - w^* \| 
  \end{align*}
  for every $n \in \N$.
\end{lemma}

\begin{proof}
  From $w^{n+1} = T_{\alpha_n f_n, w^n}(w^n)$ we obtain that
  \begin{align*}
    \| w^{n+1} - w^* \|^2
    &= \| T_{\alpha_n f_n, w^n}(w^n) - T_{\alpha_n f_n, w^n}(w^*) + T_{\alpha_n f_n, 
      w^n}(w^*) - w^*\|^2\\
    &= \| T_{\alpha_n f_n, w^n}(w^n) - T_{\alpha_n f_n, w^n}(w^*)\|^2 + \|T_{\alpha_n f_n, 
      w^n}(w^*) - w^*\|^2\\
    &\quad + 2 \dual{T_{\alpha_n f_n, w^n}(w^n) - T_{\alpha_n f_n, w^n}(w^*)}{T_{\alpha_n 
        f_n, w^n}(w^*) - w^*}\\
    &=: I_1 + I_2 + 2 I_3.
  \end{align*}
  For $I_1$, we can write
  \begin{align*}
    I_1&= \| T_{\alpha_n f_n, w^n}(w^n) - T_{\alpha_n f_n, w^n}(w^*)\|^2\\
    &= \big\| w^n - w^* + (T_{\alpha_n f_n, w^n} - I)(w^n) -
    (T_{\alpha_n f_n, w^n} - I)(w^*)\big\|^2\\
    &= \| w^n - w^*\|^2
    + 2 \dualb{w^n - w^*}{(T_{\alpha_n f_n, w^n} - I)(w^n) - (T_{\alpha_n f_n, w^n} - I)(w^*) } 
    \\
    &\quad + \| (T_{\alpha_n f_n, w^n} - I)(w^n) - (T_{\alpha_n f_n, w^n} - I)(w^*) \|^2
    =: I_{1,1} + 2I_{1,2} + I_{1,3}.
  \end{align*}
  The term $I_{1,2}$ can be estimated by applying the Cauchy--Schwarz inequality:
  \begin{align*}
    I_{1,2}&= \dualb{w^n - w^*}{ (T_{\alpha_n f_n, w^n} - I)(w^n) - (T_{\alpha_n f_n, w^n} - 
      I)(w^*) }\\
    &= - \dualB{w^n - w^*}{\frac{\alpha_n \iota \nabla f_n(w^n)}{1 + \alpha_n \|\iota 
    \nabla f_n(w^n)\|} 
      - \frac{\alpha_n \iota \nabla f_n(w^*)}{1 + \alpha_n \|\iota \nabla f_n(w^n)\|}}\\
    &\leq - \frac{\alpha_n \mu_{\xi_n} }{1 + \alpha_n \|\iota \nabla f_n(w^n)\| } \|w^n - 
    w^*\|^2.
  \end{align*}
  For $I_{1,3}$, we insert the definition of $T_{\alpha_n f_n, w^n}$ and find
  \begin{align*}
    I_{1,3} 
    &= \| (T_{\alpha_n f_n, w^n} - I)(w^n) - (T_{\alpha_n f_n, w^n} - I)(w^*) \|^2\\
    &=  \frac{\alpha_n^2 \|\iota \nabla f_n(w^n) - \iota \nabla f_n(w^*)\|^2 }{ (1 + \alpha 
    \|\iota \nabla f_n(w^n)\| )^2}.
  \end{align*}
  Thus, for $I_1$, we have
  \begin{align*}
    I_1
    \leq \Big( 1 - \frac{2 \alpha_n \mu_{\xi_n} }{1 + \alpha_n \|\iota \nabla f_n(w^n)\| } 
    \Big)\| w^n - w^*\|^2
    + \frac{\alpha_n^2 \|\iota \nabla f_n(w^n) - \iota \nabla f_n(w^*)\|^2 }{ (1 + \alpha 
      \|\iota \nabla f_n(w^n)\| )^2}.
  \end{align*}
  Further, $I_2$ can be written as
  \begin{align*}
    I_2
    = \|T_{\alpha_n f_n, w^n}(w^*) - w^*\|^2
    = \frac{\alpha_n^2 \|\iota \nabla f_n (w^*)\|^2 }{( 1 + \alpha_n \|\iota \nabla 
      f_n(w^n)\| )^2}.
  \end{align*}
  Finally, $I_3$ can be rewritten as
  \begin{align*}
    I_3 &= \dual{T_{\alpha_n f_n, w^n}(w^n) - T_{\alpha_n f_n, w^n}(w^*)}{T_{\alpha_n f_n, 
        w^n}(w^*) - w^*}\\
    &= \dual{(T_{\alpha_n f_n, w^n} - I) (w^n) - (T_{\alpha_n f_n, w^n} - I) 
      (w^*)}{(T_{\alpha_n f_n, w^n} - I) (w^*)}\\
    &\quad + \dual{w^n - w^*}{(T_{\alpha_n f_n, w^n} - I) (w^*)}
    =: I_{3,1} + I_{3,2}.
  \end{align*}
  Then for $I_{3,1}$, we insert the definition of $T_{\alpha_n f_n, w^n}$ and obtain
  \begin{align*}
    I_{3,1} &= \dual{(T_{\alpha_n f_n, w^n} - I) (w^n) - (T_{\alpha_n f_n, w^n} - I) 
      (w^*)}{(T_{\alpha_n f_n, w^n} - I) (w^*)}\\
    &\leq \frac{\alpha_n \|\iota \nabla f_n(w^n) - \iota \nabla f_n(w^*)\|}{1 + \alpha_n \|\iota 
    \nabla f_n(w^n)\|} \cdot
    \frac{\alpha_n \| \iota \nabla f_n (w^*)\|}{1 + \alpha_n \|\iota \nabla f_n(w^n)\|} \\
    &\leq \frac{1}{2}\frac{\alpha_n^2 \|\iota \nabla f_n(w^n) - \iota \nabla f_n(w^*)\|^2}{ (1 + 
    \alpha_n \|\iota \nabla f_n(w^n)\| )^2} 
    + \frac{1}{2} \frac{\alpha_n^2 \| \iota \nabla f_n (w^*)\|^2}{(1 + \alpha_n \|\iota 
      \nabla f_n(w^n)\| )^2},
  \end{align*}
  where we applied the Cauchy--Schwarz inequality and Young's inequality.
  To estimate $I_{3,2}$, we add and subtract an additional summand, so that
  \begin{align*}
    I_{3,2} &= \dual{w^n - w^*}{(T_{\alpha_n f_n, w^n} - I) (w^*)}\\
    &= \dualB{w^n - w^*}{(T_{\alpha_n f_n, w^n} - I) (w^*) + \frac{\alpha_n \iota \nabla 
    f_n(w^*)}{1 + \alpha_n \Xi }}
    - \dualB{w^n - w^*}{\frac{\alpha_n \iota \nabla f_n(w^*)}{1 + \alpha_n \Xi }},
  \end{align*}
  where $\E_{\xi_n}\big[ \dualb{w^n - w^*}{\frac{\alpha_n \iota \nabla f_n(w^*)}{1 
  + \alpha_n \Xi }} \big] = 0$ is fulfilled. Moreover, applying 
  Lemma~\ref{lem:SecondLipT} and the Cauchy--Schwarz inequality, we obtain
  \begin{align*}
    &\dualB{w^n - w^*}{(T_{\alpha_n f_n, w^n} - I) (w^*) + \frac{\alpha_n \iota \nabla 
        f_n(w^*)}{1 + \alpha_n \Xi }}\\
    &\leq \frac{ \alpha_n^2 \big|\Xi - \|\iota \nabla f_n(w^n)\| \big| \|\iota \nabla f_n(w^*)\|}
    {(1 + \alpha_n \|\iota \nabla f_n(w^n)\| ) (1 + \alpha_n \Xi )}  \| w^n 
    - w^* \|.
  \end{align*}
  Thus, the expectation of $I_3 = I_{3,1} + I_{3,2}$ can be bounded by
  \begin{align*}
    \E_{\xi_n} [I_3 ]
    &\leq \frac{1}{2} \E_{\xi_n} \Big[ \frac{\alpha_n^2 \|\iota \nabla f_n(w^n) - \iota \nabla 
    f_n(w^*)\|^2 }{\big(1 + \alpha_n \|\iota \nabla 
      f_n(w^n)\|\big)^2} + \frac{\alpha_n^2 \| \iota \nabla f_n 
      (w^*)\|^2}{ (1 + \alpha_n \|\iota \nabla f_n(w^n)\| )^2}\Big]\\
    &\quad + \E_{\xi_n} \Big[ \frac{ \alpha_n^2 \big| \Xi - \|\iota \nabla f_n(w^n)\| \big| \|\iota 
    \nabla f_n(w^*)\| }
    {(1 + \alpha_n \|\iota \nabla f_n(w^n)\| ) (1 + \alpha_n \Xi )}\Big]  \| w^n - w^* \| .
  \end{align*}
  Inserting the bounds for $I_1$, $I_2$ and $I_3$ into $\E_{\xi_n} \big[\| w^{n+1} - w^* 
  \|^2\big] = \E_{\xi_n} \big[ I_1 + I_2 + 2 I_3\big]$ finishes the proof.
\end{proof}

\begin{theorem}\label{thm:ConvNorm1}
  Let Assumption~\ref{ass:fHigherIntegrability} be fulfilled.
  For $\alpha_n = \frac{\vartheta}{n + \gamma}$, $n \in \N$, with $\gamma \in [0,\infty)$, 
  $\vartheta \in (0,\frac{1 + \gamma}{2\mu}]$ and
  \begin{align*}
    K = 2 \vartheta^2 L \mu_2 M_4^{\frac{3}{4}} 
    + \big(2 \vartheta^2 L^2 + 2 \vartheta^2 L \sigma + 2 \vartheta^2 \mu_2 \sigma 
    \big) M_2
    + 2 \vartheta^2 \sigma^2 M_2^{\frac{1}{2}} + 2 \vartheta^2 \sigma^2,
  \end{align*}
  it follows that
  \begin{align*}
    &\E_n \big[\| w^{n+1} - w^* \|^2\big]\\
    &\leq \| w_1 - w^*\|^2 (1+\gamma)^{2 \vartheta \mu}(n+1+\gamma)^{-2 \vartheta \mu} \\
    &\quad + \exptextB{\frac{2 \vartheta \mu}{1+\gamma}} K \begin{cases}
      (n+1+\gamma)^{-1} \frac{1}{2 \vartheta \mu-1}, &2 \vartheta \mu \in (1,\infty),\\
      (n+1+\gamma)^{-1} (1 + \ln{(n + \gamma)}), &2 \vartheta \mu = 1,\\
      (n+1+\gamma)^{-2 \vartheta \mu} \frac{(1+\gamma)^{2 \vartheta \mu-2}(2 \vartheta 
      \mu-2-\gamma)}{2 \vartheta \mu-1}, &2 \vartheta \mu \in [0,1),
    \end{cases}
  \end{align*}
  for every $n \in \N$.
\end{theorem}

\begin{remark}
  By choosing $\vartheta \in (\frac{1}{2\mu}, \infty)$ we obtain the optimal convergence 
  rate. A value of $\vartheta$ much larger than $\frac{1}{2\mu}$ does not improve the overall 
  rate further, but does affect the exponent in the first term of the bound that involves the initial error. 
  We note that since $\frac{1 + \gamma}{2\mu} \geq \frac{1}{2\mu}$, it is possible to make 
  this choice for any $\gamma$. 
  We could in fact instead have analyzed the simpler step size sequence with $\alpha_n = 
  \frac{\vartheta}{n}$, but chose to present the results in this form in order to match our 
  other results and comparable results for e.g.\ SGD~\cite[Theorem 
  4.7]{BottouCurtisNocedal.2018}. The results for the simpler sequence are recovered by 
  simply setting $\gamma = 0$.
\end{remark}

\begin{proof}[Proof of Theorem~\ref{thm:ConvNorm1}]
  The main idea of the proof is to apply the bound from Lemma~\ref{lem:boundForTheo} 
  with $\Xi = 0$ and bound the denominators of the last three 
  summands from below by one. We then get
  \begin{align*}
    &\E_{\xi_n} \big[\| w^{n+1} - w^* \|^2\big]\\
    &\leq  \Big( 1 - \E_{\xi_n} \Big[ \frac{2 \alpha_n \mu_{\xi_n} }{1 + \alpha_n \|\iota \nabla 
      f_n(w^n)\| } \Big] \Big) \| w^n - w^*\|^2
    + 2 \alpha_n^2 L^2 \| w^n - w^*\|^2
    + 2 \alpha_n^2 \sigma^2 \\
    &\quad + 2 \alpha_n^2 \big(\E_{\xi_n} \big[ \|\iota \nabla f_n(w^n)\|^2 
    \big]\big)^{\frac{1}{2}} \sigma \| w^n - w^* \| \\
    &\leq \Big(1 - \E_{\xi_n} \Big[ \frac{2 \alpha_n \mu_{\xi_n} }{1 + \alpha_n L_{\xi_n}
    \|w^n - w^n\| + \alpha_n \|\iota \nabla f_n(w^*)\|} \Big] \Big) \| w^n - w^*\|^2\\
    &\quad + \big(2 \alpha_n^2 L^2 + 2 \alpha_n^2 L \sigma \big)\| w^n - w^*\|^2 
    + 2 \alpha_n^2 \sigma^2 \| w^n - w^* \| + 2 \alpha_n^2 \sigma^2 \\
    &\leq \Big(1 - \E_{\xi_n} \Big[ \frac{2 \alpha_n \mu_{\xi_n} }{1 + \alpha_n \|\iota 
      \nabla f_n(w^*)\|} + \frac{2 \alpha_n^2 L_{\xi_n} \mu_{\xi_n} }{(1 + \alpha_n \|\iota 
      \nabla f_n(w^*)\|)^2} \| w^n - w^*\| \Big]\Big) \| w^n - w^*\|^2\\
    &\quad + \alpha_n^2 \big(\big(2 L^2 + 2 L \sigma \big) \| w^n - w^*\|^2 
    + 2 \sigma^2 \| w^n - w^* \| + 2 \sigma^2\big) \\  
    &\leq \Big(1 - \E_{\xi_n} \Big[  \frac{2 \alpha_n \mu_{\xi_n} }{1 + \alpha_n \|\iota 
      \nabla f_n(w^*)\|} \Big]\Big) \| w^n - w^*\|^2\\
    &\quad +  \alpha_n^2 \big(2 L \mu_2 \| w^n - w^*\|^3 + \big(2 
    L^2 + 2 L \sigma \big) \| w^n - w^*\|^2
    + 2 \sigma^2 \| w^n - w^* \| + 2 \sigma^2 \big),
  \end{align*}
  where we added and subtracted $\iota \nabla f_n(w^*)$ and applied Minkowski's 
  inequality in the second step and Lemma~\ref{lem:TaylorEx} in the third.
  For the first summand of the previous inequality, we apply Lemma~\ref{lem:TaylorEx} 
  once more and find
  \begin{align*}
    & \Big(1 - \E_{\xi_n} \Big[ \frac{2 \alpha_n \mu_{\xi_n} }{1 + \alpha_n \|\iota 
      \nabla f_n(w^*)\|} \Big]\Big) \| w^n - w^*\|^2\\
    &= \Big(1 - \E_{\xi_n} \Big[ \frac{2 \vartheta \mu_{\xi_n} }{n + \gamma + \vartheta \|\iota 
    \nabla f_n(w^*)\|} \Big]\Big) \| w^n - w^*\|^2\\
    &\leq \Big(1 - \E_{\xi_n} \Big[ \frac{2 \vartheta \mu_{\xi_n}}{n + \gamma}\Big]
    + \E_{\xi_n} \Big[ \frac{2 \vartheta \mu_{\xi_n}}{(n + \gamma)^2} \vartheta \|\iota \nabla 
    f_n(w^*)\| \Big]\Big)
    \| w^n - w^*\|^2\\
    &\leq \Big(1 - \frac{2 \vartheta \mu }{n + \gamma} \Big) \| w^n - w^*\|^2
    + 2 \alpha_n^2 \mu_2 \sigma \| w^n - w^*\|^2.
  \end{align*}
  Taking the $\E_{n-1}$-expectation and applying the a priori bounds from 
  Lemma~\ref{lem:apriori2} and Lemma~\ref{lem:apriori4}, 
  we find that
  \begin{align*}
    &\E_n \big[\| w^{n+1} - w^* \|^2\big]\\
    &\leq \Big(1 - \frac{2 \vartheta \mu }{n + \gamma} \Big) \E_{n-1} \big[ \| w^n - w^*\|^2\big]
    + 2 \alpha_n^2 L \mu_2 \E_{n-1} \big[\| w^n - w^*\|^3 \big]\\
    &\quad  + \big(2 \alpha_n^2 L^2 + 2 \alpha_n^2 
    L \sigma + 2 \alpha_n^2 \mu_2 \sigma \big) \E_{n-1}\big[\| w^n - 
    w^*\|^2\big]\\
    &\quad + 2 \alpha_n^2 \sigma^2 \E_{n-1}\big[ \| w^n - w^* \|\big] + 2 \alpha_n^2 
    \sigma^2\\
    &\leq \Big(1 - \frac{2 \vartheta \mu }{n + \gamma} \Big) \E_{n-1} \big[ \| w^n - w^*\|^2\big]
    + 2 \alpha_n^2 L \mu_2 M_4^{\frac{3}{4}} \\
    &\quad + \big(2 \alpha_n^2 L^2 + 2 \alpha_n^2 L \sigma + 2 \alpha_n^2 \mu_2 \sigma 
    \big) M_2
     + 2 \alpha_n^2 \sigma^2 M_2^{\frac{1}{2}} + 2 \alpha_n^2 \sigma^2\\
    &= \Big(1 - \frac{2 \vartheta \mu }{n + \gamma} \Big) \E_{n-1} \big[ \| w^n - w^*\|^2\big]
    + \frac{K}{(n + \gamma)^2}.
  \end{align*}
  Reinserting the inequality $n-1$ times yields
  \begin{align*}
    &\E_n \big[\| w^{n+1} - w^* \|^2\big]\\
    &\leq \| w_1 - w^*\|^2\prod_{i=1}^{n} \Big( 1 - \frac{2 \vartheta \mu}{\gamma + i} \Big)
    + K \sum_{i=1}^{n} \frac{1}{(\gamma+ i)^2} \prod_{j=i+1}^{n} \Big( 1 - \frac{2 \vartheta 
    \mu}{\gamma + j} \Big).
  \end{align*}
  Due to the assumption $\vartheta \in (0,\frac{1 + \gamma}{2 \mu }]$, we can apply 
  Lemma~\ref{lem:algebraic_inequalities} with $x=2 \vartheta \mu$ and 
  $y=\gamma$ and obtain the claimed error bound.
\end{proof}

In comparison to convergence results regarding SGD, e.g.\ 
\cite[Theorem~4.7]{BottouCurtisNocedal.2018}, the higher-moment bounds in 
Assumption~\ref{ass:fHigherIntegrability} 
are not necessary. These are in fact not needed to prove convergence of TSGD, as the following theorem demonstrates.
The drawback is that the contraction parameter is given by $1 - \frac{2 \alpha_n 
\mu_{\xi_n} }{1 + \alpha_n \|\iota \nabla f_n(w^n)\| }$, where we cannot verify that $\|\iota 
\nabla f_n(w^n)\|$ is bounded. Thus, it is not necessarily possible to prove the optimal 
rate of convergence in all cases. We note that the step size sequence involving 
$\gamma$ is important here, as it allows us to choose a large $\vartheta$ for which the 
parameter $C$ defined in the theorem below becomes as large as possible, leading to the 
best possible rate. A larger $\gamma$ also increases the error term arising from the initial 
error, but as argued in~\cite[p.~251]{BottouCurtisNocedal.2018} the influence of this term 
can be minimized by precomputing a better $w_1$ using e.g.\ TSGD with a constant step 
size. We note that the condition $C \le 1 + \gamma$ is mostly technical, will likely not be 
an issue in practice, and can always be satisfied by choosing $\gamma \ge 1$.

\begin{theorem}\label{thm:ConvNorm2}
  Let Assumption~\ref{ass:fStoch} be fulfilled.
  For $\alpha_n = \frac{\vartheta}{n + \gamma}$, $n \in \N$, with $\gamma \in [0,\infty)$, 
  $\vartheta \in (0,\infty)$ such that
  \begin{align*}
    C &= \E_{\xi} \Big[ \frac{2 \gamma \vartheta \mu_{\xi} 
    }{\gamma + \vartheta L_{\xi} \| w_1 - w^*\| + \gamma \vartheta L_{\xi} + \vartheta \|\iota 
      \nabla f(\xi,w^*)\|} \Big],\\
    K &= 2 \vartheta^2 \big(\big( L^2 + L \sigma \big) M_2 + \sigma^2 + \sigma^2 
    M_2^{\frac{1}{2}} \big)
  \end{align*}
  and $C \in (0,1 + \gamma]$ it follows that
  \begin{align*}
    &\E_n \big[\| w^{n+1} - w^* \|^2\big]\\
    &\leq \| w_1 - w^*\|^2 (1+\gamma)^C(n+1+\gamma)^{-C} \\
    &\quad + \exptextB{\frac{C}{1+\gamma}} K \begin{cases}
      (n+1+\gamma)^{-1} \frac{1}{C-1}, &C \in (1,\infty),\\
      (n+1+\gamma)^{-1} (1 + \ln{(n + \gamma)}), &C = 1,\\
      (n+1+\gamma)^{-C} \frac{(1+\gamma)^{C-2}(C-2-\gamma)}{C-1}, &C \in [0,1),
    \end{cases}
  \end{align*}
  for every $n \in \N$.
\end{theorem}

\begin{proof}
  As in the proof of Theorem~\ref{thm:ConvNorm1}, the main idea of the proof is to apply 
  the bound from Lemma~\ref{lem:boundForTheo} 
  with $\Xi = 0$, where we also bound the denominators of the last three 
  summands from below by one. We then get
  \begin{align*}
    &\E_{\xi_n} \big[\| w^{n+1} - w^* \|^2\big]\\
    &\leq \Big( 1 - \E_{\xi_n} \Big[ \frac{2 \alpha_n \mu_{\xi_n} }{1 + \alpha_n \|\iota \nabla 
      f_n(w^n)\| } \Big] \Big) \| w^n - w^*\|^2
    + 2 \alpha_n^2 L^2 \| w^n - w^*\|^2
    + 2 \alpha_n^2 \sigma^2 \\
    &\quad + 2 \alpha_n^2 \big(\E_{\xi_n} \big[ \|\iota \nabla f_n(w^n)\|^2 
    \big]\big)^{\frac{1}{2}} \sigma \| w^n - w^* \| \\
    &\leq \Big( 1 - \E_{\xi_n} \Big[ \frac{2 \alpha_n \mu_{\xi_n} }{1 + \alpha_n \|\iota \nabla 
      f_n(w^n)\| } \Big] \Big) \| w^n - w^*\|^2
    + 2 \alpha_n^2 (L^2 + L \sigma) \| w^n - w^*\|^2\\
    &\quad + 2 \alpha_n^2 \sigma^2 
    + 2 \alpha_n^2 \sigma^2 \| w^n - w^* \|\\
    &=: \Big( 1 - \E_{\xi_n} \Big[ \frac{2 \alpha_n \mu_{\xi_n} }{1 + \alpha_n \|\iota \nabla 
      f_n(w^n)\| } \Big] \Big) \| w^n - w^*\|^2 + \alpha_n^2 I,
  \end{align*}
  where we added and subtracted $\iota \nabla f_n(w^*)$ and applied Minkowski's 
  inequality in the second step.
  Using the a priori bound from 
  Lemma~\ref{lem:apriori2}, we find that
  \begin{align*}
    \alpha_n^2 \E_{n-1} [I]
    &= 2 \alpha_n^2 \big(\big(L^2 + L \sigma \big) \E_{n-1} \big[\| w^n - w^*\|^2\big]
    + \sigma^2 + \sigma^2 \E_{n-1} \big[\| w^n - w^* \|\big]\big) \\
    &\leq \frac{2 \vartheta^2}{(n + \gamma)^2} \big(\big( L^2 + L \sigma \big) M_2 + 
    \sigma^2 + \sigma^2 
    M_2^{\frac{1}{2}}\big) = \frac{K}{(n + \gamma)^2}.
  \end{align*}
  Applying the pathwise a priori bound from Lemma~\ref{lem:aprioriPath}, it follows that
  \begin{align*}
    \alpha_n \|\iota \nabla f_n(w^n)\| 
    &\leq \alpha_n L_{\xi_n} \| w^n - w^*\| + \alpha_n \|\iota \nabla f_n(w^*)\| \\
    &\leq \frac{\vartheta}{n + \gamma} L_{\xi_n} \| w_1 - w^*\| + \frac{\vartheta n}{\gamma + 
      n} L_{\xi_n} + \frac{\vartheta}{n + \gamma} \|\iota \nabla f_n(w^*)\| \\
    &\leq \frac{\vartheta}{\gamma} L_{\xi_n} \| w_1 - w^*\| + \vartheta
    L_{\xi_n} + \frac{\vartheta}{\gamma} \|\iota \nabla f_n(w^*)\|.
  \end{align*}
  Thus, we find 
  \begin{align*}
    &1 - \E_{\xi_n} \Big[\frac{2 \alpha_n \mu_{\xi_n}}{1 + \alpha_n \|\iota \nabla 
      f_n(w^n)\| }\Big] \\
    &\leq 1 - \frac{1}{n + \gamma} \cdot \E_{\xi} \Big[ \frac{2 \gamma \vartheta \mu_{\xi} 
    }{\gamma + \vartheta L_{\xi} \| w_1 - w^*\| + \gamma \vartheta L_{\xi} + \vartheta \|\iota 
      \nabla f(\xi,w^*)\|} \Big] 
    =  1 - \frac{C}{n + \gamma}.
  \end{align*}
  Altogether, this implies that
  \begin{align*}
    &\E_n \big[\| w^{n+1} - w^* \|^2\big]\\
    &\leq \| w_1 - w^*\|^2\prod_{i=1}^{n} \Big( 1 - \frac{C}{\gamma + i} \Big)
    + K \sum_{i=1}^{n} \frac{1}{(\gamma+ i)^2} \prod_{j=i+1}^{n} \Big( 1 - \frac{C}{\gamma + 
      j} \Big).
  \end{align*}
  is fulfilled. As $C \in (0,1 + \gamma]$ is fulfilled by assumption, the claim of the 
  theorem can be verified by an application of Lemma~\ref{lem:algebraic_inequalities} with 
  $x=C$ and $y=\gamma$.
\end{proof}

In the penultimate theorem, we prove a convergence result under the additional assumption that 
the gradient is bounded. Note that the error bound does not increase uncontrollably with 
growing $\gamma$ or $\vartheta$. The terms $\gamma$ and
$\vartheta B$ always appear with a positive exponent in one factor and with the same, but 
negative, exponent in another factor. This verifies that TSGD is very stable with respect 
to large initial step sizes.

\begin{theorem}\label{thm:ConvNorm3}
  Let Assumption~\ref{ass:fBounded}
  be fulfilled. 
  For $\alpha_n = \frac{\vartheta}{n + \gamma}$, $n \in \N$ with $\gamma \in [0,\infty)$, 
  $1 + \gamma \geq \vartheta \big(2 \mu -  B\big)$ and $K = \big(4 + 6D^2\big) B^2 + 2 
  \big(B^2 D + \sigma B\big) M_2^{\frac{1}{2}}$ 
  it follows that
  \begin{align*}
    &\E_n \big[\| w^{n+1} - w^* \|^2\big]\\
    &\leq \| w_1 - w^*\|^2 (1+\gamma + 
    \vartheta B)^{2 \vartheta \mu} (n+1+\gamma + \vartheta B)^{-2 \vartheta \mu} \\
    &\quad + \exptextB{\frac{2 \mu \vartheta}{1 + \gamma+ \vartheta B}} \vartheta^2 K 
    \times\\
    &\quad \begin{cases}
      (n+1+\gamma+\vartheta B)^{-1} \frac{1}{2 \vartheta \mu-1}, &2 \vartheta \mu \in 
      (1,\infty),\\
      (n+1+\gamma+\vartheta B)^{-1}\big(1 + \ln{(n+\gamma+\vartheta B)} \big), &2 
      \vartheta \mu = 1,\\
      (n+1+\gamma+\vartheta B)^{-2 \vartheta \mu} \frac{ (1 + \gamma + \vartheta B)^{2 
          \vartheta \mu -2} (2 \vartheta \mu - 2 - \gamma+\vartheta B )}{2 \vartheta 
          \mu - 1}, &2 \vartheta \mu \in [0,1)
    \end{cases}
  \end{align*}
  for every $n \in \N$.
\end{theorem}

\begin{proof}
  Again, we apply the bound from Lemma~\ref{lem:boundForTheo} but this time with $\Xi 
  = B$ to acquire
  \begin{align*}
    &\E_{\xi_n} \big[\| w^{n+1} - w^* \|^2\big]\\
    &\leq \Big( 1 - \E_{\xi_n} \Big[ \frac{2 \alpha_n \mu_{\xi_n}}{1 + \alpha_n \|\iota \nabla 
      f_n(w^n)\| } \Big] \Big) \| w^n - w^*\|^2\\
    &\quad + 4 \E_{\xi_n} \Big[ \frac{\alpha_n^2 \|\iota \nabla f_n(w^n) \|^2}{(1 + \alpha_n 
    \|\iota \nabla f_n(w^n)\| )^2} \Big]
      + 6 \E_{\xi_n} \Big[ \frac{\alpha_n^2 \|\iota \nabla f_n (w^*)\|^2 }{( 1 
      + \alpha_n \|\iota \nabla f_n(w^n)\| )^2} \Big]\\
    &\quad + 2\E_{\xi_n} \Big[ \frac{ \alpha_n^2 \big(B + \|\iota  \nabla f_n(w^n)\| \big) \|\iota 
    \nabla f_n(w^*)\| }
    {(1 + \alpha_n \|\iota \nabla f_n(w^n)\| ) (1 + \alpha_n B )} \Big] 
    \| w^n - w^* \| \\
    &\leq \Big( 1 - \E_{\xi_n} \Big[\frac{2 \alpha_n \mu_{\xi_n}}{1 + \alpha_n B }\Big] \Big) \| 
    w^n - w^*\|^2
    + \frac{4\alpha_n^2 B^2}{(1 + \alpha_n B)^2}\\
    &\quad + 6 \E_{\xi_n} \Big[ \frac{\alpha_n^2 \|\iota \nabla f_n (w^*)\|^2 }{( 1 
    + \alpha_n \|\iota \nabla f_n(w^n)\| )^2} \Big]\\
    &\quad + 2 \Big(\frac{ \alpha_n B}{1 + \alpha_n B} \E_{\xi_n} \Big[ \frac{ \alpha_n \|\iota 
    \nabla f_n(w^*)\|}{1 + \alpha_n \|\iota \nabla f_n(w^n)\| }\Big] 
    + \frac{ \alpha_n^2 \sigma B}{(1 + \alpha_n B)^2}\Big) \| w^n - w^* \|,
  \end{align*}
  where we used in the last step that the function $x\mapsto \frac{x}{1 + x}$ is monotonically increasing for $x \in [0,\infty)$.
  We also have
  \begin{align*}
    \E_{\xi_n} \Big[ \frac{\alpha_n^2 \|\iota \nabla f_n (w^*)\|^2 }{( 1 + \alpha_n \|\iota 
      \nabla f_n(w^n)\| )^2} \Big]
    &\leq \E_{\xi_n} \Big[ \frac{\|\iota \nabla f_n (w^*)\|^2 }{\|\iota \nabla f_n(w^n)\|^2} 
    \cdot \frac{\alpha_n^2 \|\iota \nabla f_n (w^n)\|^2 }{( 1 + \alpha_n \|\iota 
      \nabla f_n(w^n)\| )^2} \Big]\\
    &\leq \frac{\alpha_n^2 B^2 D^2 }{( 1 + \alpha_n B)^2},
  \end{align*}
  and analogously $\E_{\xi_n} \big[ \frac{\alpha_n \|\iota \nabla f_n (w^*)\| }{1 + \alpha_n 
  \|\iota \nabla f_n(w^n)\|} \big] \leq \frac{\alpha_n B D}{1 + \alpha_n B}$.
  It then follows that
  \begin{align*}
    \E_{\xi_n} \big[\| w^{n+1} - w^* \|^2\big]
    &\leq \Big( 1 - \frac{2 \alpha_n \mu}{1 + \alpha_n B } \Big) \| w^n - w^*\|^2
    + \frac{\alpha_n^2 B^2 \big(4 + 6D^2\big) }{(1 + \alpha_n B)^2}\\
    &\quad + \frac{ 2 \alpha_n^2 B^2 D + 2 \alpha_n^2 \sigma B}{(1 + \alpha_n B)^2} 
    \| w^n - w^* \|,
  \end{align*}
  and taking the $\E_{n-1}$-expectation, we find that
  \begin{align*}
    &\E_n \big[\| w^{n+1} - w^* \|^2\big]\\
    &\leq \Big( 1 - \frac{2 \alpha_n \mu}{1 + \alpha_n B } \Big) \E_{n-1} \big[ \| w^n - 
    w^*\|^2\big]
    + \frac{\alpha_n^2 B^2 \big(4 + 6D^2\big) }{(1 + \alpha_n B)^2}\\
    &\quad + \frac{ 2 \alpha_n^2 B^2 D + 2 \alpha_n^2 \sigma B}{(1 + \alpha_n B)^2} 
    \big(\E_{n-1} \big[ \| w^n - w^*\|^2 \big]\big)^{\frac{1}{2}}\\
    &\leq \Big( 1 - \frac{2 \alpha_n \mu}{1 + \alpha_n B } \Big) \E_{n-1} \big[ \| w^n - 
    w^*\|^2\big]
    + \frac{\alpha_n^2 B^2 \big(4 + 6D^2\big) }{(1 + \alpha_n B)^2}
    + \frac{ 2 \alpha_n^2 \big(B^2 D + \sigma B\big) M_2^{\frac{1}{2}}}{(1 + 
    \alpha_n B)^2} \\
    &= \Big( 1 - \frac{2 \vartheta \mu}{n + \gamma + \vartheta B } \Big) \E_{n-1} \big[ \| 
    w^n 
    - w^*\|^2\big] + \frac{\vartheta^2 K}{(n + \gamma + \vartheta B)^2}.
  \end{align*}
  Reinserting the bound $n-1$ times, it follows that
  \begin{align*}
    \E_n \big[\| w^{n+1} - w^* \|^2\big]
    &\leq \| w_1 - w^*\|^2\prod_{i=1}^{n} \Big(1 - \frac{2 \vartheta \mu}{i + \gamma + 
    \vartheta B }\Big)\\
    &\quad + \vartheta^2 K \sum_{i=1}^{n} \frac{1}{(i + \gamma + \vartheta B)^2} 
    \prod_{j=i+1}^{n} \Big(1 - \frac{2 \vartheta \mu}{j + \gamma + \vartheta B }\Big).
  \end{align*}
  Due to the restriction $n +\gamma \geq \vartheta \big(2 \mu -  B\big)$, we can now apply 
  Lemma~\ref{lem:algebraic_inequalities} with $x=2\vartheta\mu$ and 
  $y=\gamma + \vartheta B$ in order to finish the proof of the theorem.
\end{proof}

We note that convergence results in this area are often stated in the form $F(w^{n}) - 
F(w^*) \le \frac{C}{n}$. The above theorems are a slightly stronger version, in that it 
proves convergence of the iterates themselves. However, in our setting they are 
equivalent as the following theorem shows.
We note that it is possible to use an approach similar to the one above to directly prove 
the convergence of $\{F(w^n)\}_{n \in \N}$. The error constants thereby acquired are 
similar to those acquired from a combination of one of the theorems above and 
Theorem~\ref{thm:Fconv} below. However, the reverse 
approach of proving convergence of $\{w^n\}_{n \in \N}$ by using convergence of 
$\{F(w^n)\}_{n \in \N}$ results in an additional factor $\frac{2}{\mu}$ which is typically very 
large.

\begin{theorem}\label{thm:Fconv}
  Let Assumption~\ref{ass:fStoch} be fulfilled. Then $\big\{\E_n 
  \big[\|w^{n+1} - w^*\|^2\big]\big\}_{n \in \N}$ behaves asymptotically the same as $\big\{ 
  \E_n \big[F(w^{n+1}) \big] - F(w^*) \big\}_{n \in \N}$. More precisely, 
  \begin{align*}
    \E_n \big[F(w^{n+1}) \big] - F(w^*)
    &\leq \frac{L}{2}  \E_n \big[\|w^{n+1} - w^*\|^2\big], \quad \text{and}\\
    \E_n \big[\|w^{n+1} - w^*\|^2\big]
    &\leq \frac{2}{\mu} \E_n \big[F(w^{n+1}) \big] - F(w^*)
  \end{align*}
  are fulfilled for every $n \in \N$.
\end{theorem}

\begin{proof}
  By applying Lemma~\ref{lem:Fconsequences}, with $v = w^{n+1}$ and $w = w^*$ we 
  find that 
  \begin{equation*}
    F(w^{n+1}) \leq F(w^*) + \dual{\iota \nabla F(w^*) }{ w^{n+1} - w^*} + \frac{L}{2} \|w^{n+1} 
    - w^*\|^2.
  \end{equation*}
  But since $\nabla F(w^*) = 0$, this directly implies
  \begin{align*}
    \E_n \big[F(w^{n+1}) \big] - F(w^*) \leq \frac{L}{2}  \E_n \big[\|w^{n+1} - w^*\|^2\big].
  \end{align*}
  As $F$ is strongly convex, it follows that
  \begin{align*}
    \E_n \big[F(w^{n+1})\big] - F(w^*) 
    &\geq \E_n \big[\dual{\iota \nabla F(w^*) }{ w^{n+1} - w^*}\big] + \frac{\mu}{2} \E_n 
    \big[\|w^{n+1} - w^*\|^2\big]\\
    &= \frac{\mu}{2} \E_n \big[ \|w^{n+1} - w^*\|^2\big],
  \end{align*}
  since $\nabla F(w^*) = 0$ which verifies the second inequality.
\end{proof}

\section{Numerical experiments} \label{section:experiments}

In this section, we illustrate our theoretical results and the advantages of the TSGD 
method by performing a few numerical experiments. We consider binary classification, 
which means 
that we have $N \in \N$ given data samples $x_i \in \R^d$ and corresponding labels $y_i 
\in \{0,1\}$, $i \in \{1,\dots,N\}$. Each $x_i$ belongs to one of two classes; to the first one 
if $y_i = 0$ and to the second if $y_i = 1$. The goal is to find a prediction function $h_w 
\colon \R^d \to \R$ such that $h_w(x_i) \approx y_i$ for every $i \in 
\{1,\dots,N\}$. The prediction function depends on the parameters $w \in \R^{n_w}$ 
and is of a specific, given type. 
Here, we consider two different types; the first is a support vector machine (SVM) where 
$h_w(x) = \dual{\hat{w}}{x} + b$ for $w = (\hat{w},b)$. This is an affine classifier, which fits 
into our analysis. The second type is a general fully connected neural 
network~\cite{HighamHigham.2019} where $h_w(x)$ depends nonlinearly on the 
parameters $w$. This type of classifier does not fit directly into our analysis, but the 
TSGD method still performs well.

To measure the performance of the classifier we use the log loss function $\ell \colon 
\R^2 \to \R$ given by $\ell(h_w(x), y) = \ln(1 + \exptext{-h_w(x)y})$. We also add a 
regularization term $\frac{\lambda}{2} \|w\|^2$ with $\lambda \in (0,\infty)$, which makes 
the problem strongly convex in the SVM case. The overall problem is then to minimize 
the empirical risk $F(w)$, where
\begin{align*}
  F(w) = \frac{1}{N} \sum_{i = 1}^{N} \ell(h_w(x_i), y_i) +  \frac{\lambda}{2} \|w\|^2.
\end{align*}

We choose two different data sets from the LIBSVM collection%
\footnote{Hosted at \url{https://www.csie.ntu.edu.tw/~cjlin/libsvmtools/datasets/}.},
namely the \textit{mushroom} data set (originally from the {UCI} Machine Learning 
Repository~\cite{DuaGraff.2019}\footnote{Available at 
\url{https://archive.ics.uci.edu/ml/datasets/mushroom}.})
and the \textit{rcv1.binary} data set~\cite{RCV1.2004}\footnote{Available at \url{https://www.csie.ntu.edu.tw/~cjlin/libsvmtools/datasets/binary.html\#rcv1.binary}.}.
The former has $N = 8124$ samples with $d = 112$ features while the latter 
contains $N = 20242$ samples with $d = 47236$ features each.

The regularization parameter $\lambda$ corresponds to the convexity parameter $\mu$ 
from Assumption~\ref{ass:fStoch}. The other parameters $L$ and $\sigma$ 
appearing in our theory are more difficult to state explicitly. As our theoretical
results show that the choice of TSGD step size does not depend on these parameters,
this is not an issue for TSGD.
In comparison, for the step size sequence $\{\alpha_n\}_{n \in \N}$ with 
$\alpha_n = \frac{\vartheta}{n + \gamma}$, the optimal choice for the parameter 
$\gamma$ for SGD can depend on $\frac{1}{L}$, 
see~\cite[Theorem~4.7]{BottouCurtisNocedal.2018}.
This makes it more difficult to find a suitable initial step size.

We have implemented both the SGD and TSGD methods in Python. Due to the low 
complexity of the methods, this is fairly straightforward, and the main issue is how to 
compute the gradients $\nabla f_n(w^n)$, $n \in \N$. In the SVM case, this is also 
straightforward, and we can directly write down a closed-form expression that depends on 
the data $x_i$, $i \in \{1,\dots,N\}$. In the neural network case, we rely on the scikit-learn 
library~\cite{Scikit-learn} and its backpropagation implementation. In both cases, we 
deviate slightly from the presented analysis in that we do not choose the batches 
completely randomly. Instead, we follow the conventional procedure of splitting the data 
set into a number of batches and picking from these without replacement. When there are 
none left, the data is reshuffled and new batches are created. One such sequence is 
referred to as an epoch. 

In the examples, we compare the TSGD method with the classical SGD method. We plot 
the errors $\E_n\big[F(w^{n+1})\big] - F(w^*)$, which according to 
Theorem~\ref{thm:Fconv} behave similarly to the errors $\|w^{n+1}-w^*\|^2$ for the 
number of steps $n \in \N$. As we only prove the convergence in 
expectation, we use $100$ sample paths and plot the average error for them.
As the mushroom data set is comparably small, we can compute a reference solution 
$F(w^*)$ by using the nonlinear equation solver provided by the package 
\texttt{scipy.optimize}~\cite{SciPy} in the SVM setting. This enables us to show the exact 
values $F(w^n) - F(w^*)$. In all the other examples, 
we compute a reference solution $F(w^*)$ by simply running the TSGD scheme for more 
steps and choosing the lowest value obtained during all steps. The $F(w^*)$ thereby 
acquired is not the exact minimum but a very good approximation thereof. 
We choose TSGD to obtain the reference solution $F(w^*)$ as smaller 
values $F(w^n)$, $n \in \N$, are obtained using this method and therefore the obtained 
value $F(w^*)$ is as small as possible.

In the following two sub-sections we further describe parameter choices and the results of the different settings.

\subsection{Support vector machine}
We used a batch size of $1 \%$ of the amount of samples for both data sets. 
The regularization parameter was chosen as $\lambda = 10^{-5}$.
We ran the example for $10$ epochs but only stored every tenth value $F(w^n)$ in order 
to save computational costs.
For the step size $\alpha_n = \frac{\vartheta}{n + \gamma}$ we chose $\vartheta = 2 \cdot 
10^5 = 2 \lambda^{-1}$ in order to ensure the optimal speed of convergence of TSGD 
from Theorems~\ref{thm:ConvNorm1} and \ref{thm:ConvNorm3} and to fit the 
restriction for SGD from \cite[Theorem~4.7]{BottouCurtisNocedal.2018}.
Further, we varied $\gamma = 10^m$ for $m \in \{0, \dots, 6\}$ to investigate how larger 
initial step sizes effect the errors. Note that in 
\cite[Theorem~4.7]{BottouCurtisNocedal.2018} there is also a lower bound for $\gamma$. 
This restriction cannot be stated easily as it depends for example on the Lipschitz 
constant of $\nabla f_n$. The optimal rate can be observed for $\gamma$ large enough in 
the SVM examples.

In Figure~\ref{fig:lossSVM_mush}, we observe very well how larger initial step sizes 
change the outcome. For the TSGD method, we see how the error decreases while 
decreasing $\gamma$ within $\{10^3, 10^4, 10^5, 10^6\}$ and thereby increasing the 
initial step size. 
When $\gamma$ is chosen within $\{1, 10, 10^2, 10^3\}$ the error stops decreasing but 
remains within the same ballpark. 
This behavior cannot be observed for SGD. While increasing the initial step size 
has a positive effect for $\gamma$ between $\{10^4, 10^5, 10^6\}$, it leads to large errors 
within the first few steps for $\gamma = 10^{m}$ for 
$m \in \{0,1,2,3\}$ that can no longer be compensated for at later points. We note that we 
observe a faster asymptotic convergence for TSGD than suggested by our bounds 
(although we acknowledge that choosing a representative reference curve is a non-trivial 
task, given the number of different results in the plots).
A possible explanation could be that the error in 
Theorems~\ref{thm:ConvNorm1}, \ref{thm:ConvNorm2} and 
\ref{thm:ConvNorm3}  consists of two parts where the exponent in the second 
summand cannot be smaller than $-1$. 
In our case it could be the first error part that is dominating the total error. Here, the error 
can decrease faster than $\sim n^{-1}$ for large $\vartheta$.
The second part of the error corresponds to the question how well the operator $T_{\alpha 
f_n, z}(w)= w - \frac{\alpha \iota \nabla f_n(w)}{1 + \alpha \|\iota \nabla f_n(z)\|}$ preserves 
the optimum $w^*$. In the deterministic case, i.e.~ $f_n = F$, it follows that $T_{\alpha F, 
z}(w^*) = w^*$. Thus, the stochastic approximation of $F$ by $f_n$ could be 
better than expected in our examples.

In Figure~\ref{fig:lossSVM_cat}, we observe similar results for the second data set. 

\begin{figure}
  \includegraphics[width=0.49\textwidth]{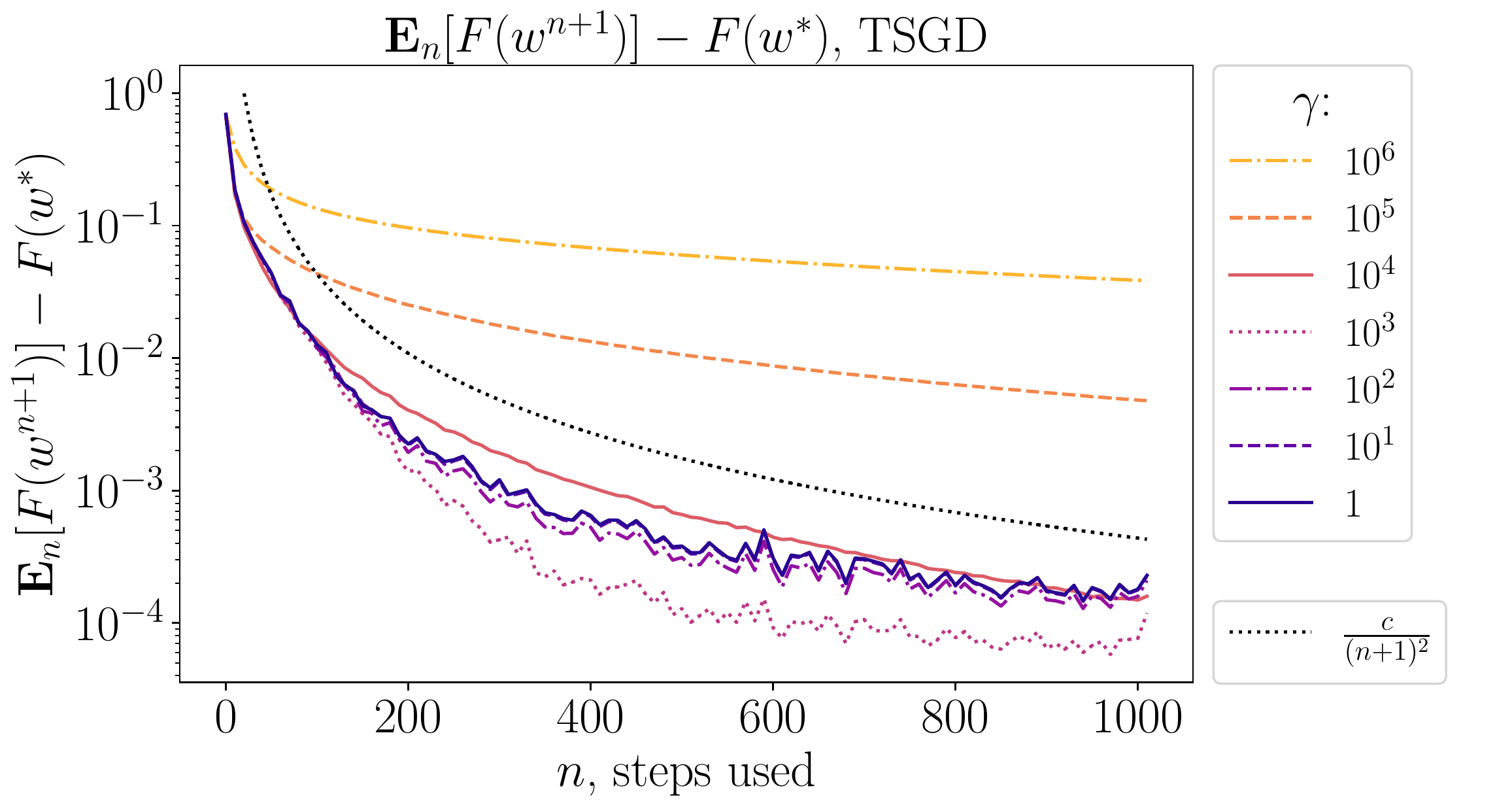}
  \includegraphics[width=0.49\textwidth]{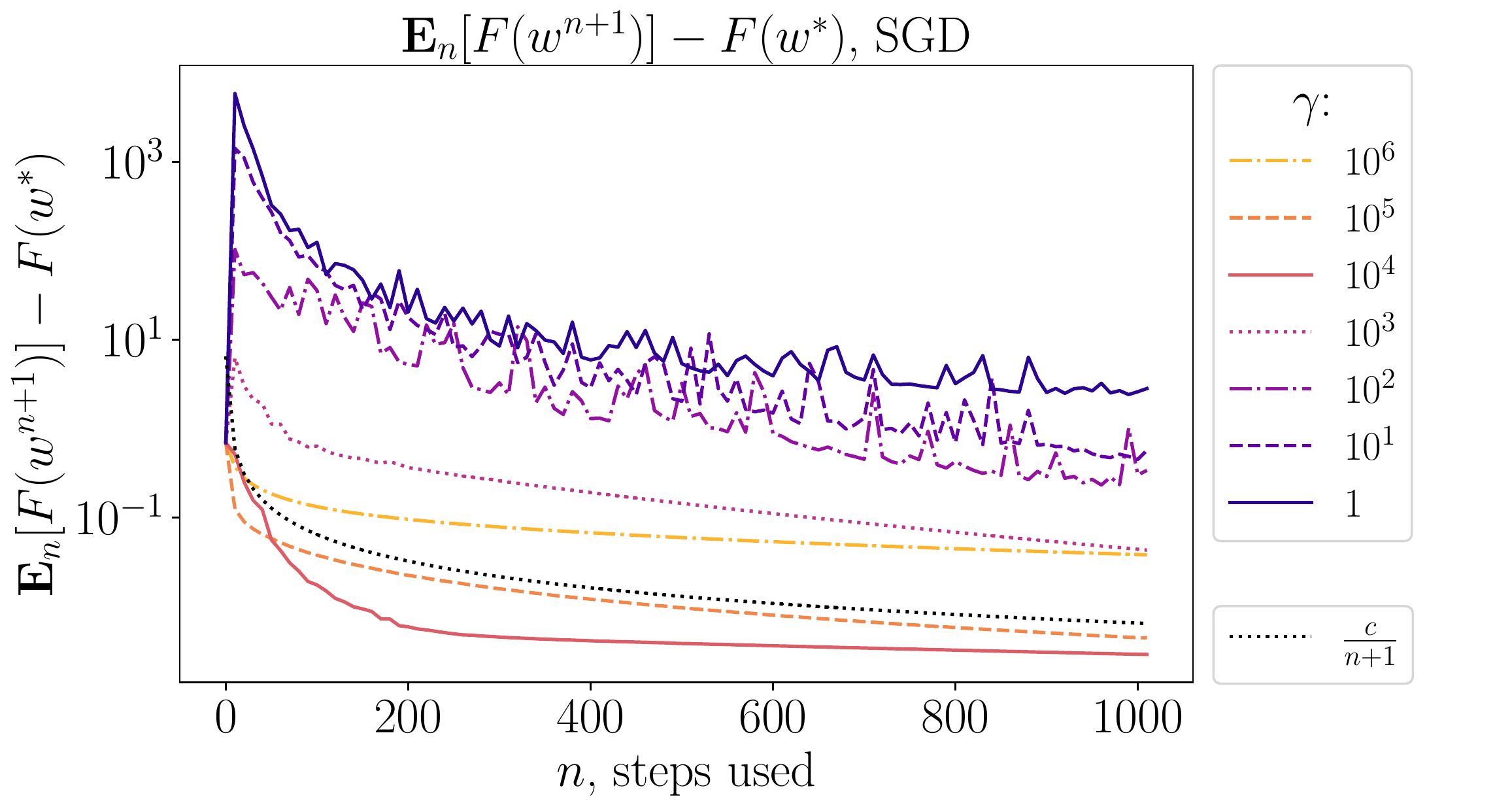}
  \caption{TSGD (left) versus SGD (right), different step sizes, SVM, 
    mushrooms}
  \label{fig:lossSVM_mush}
\end{figure}

\begin{figure}
  \includegraphics[width=0.49\textwidth]{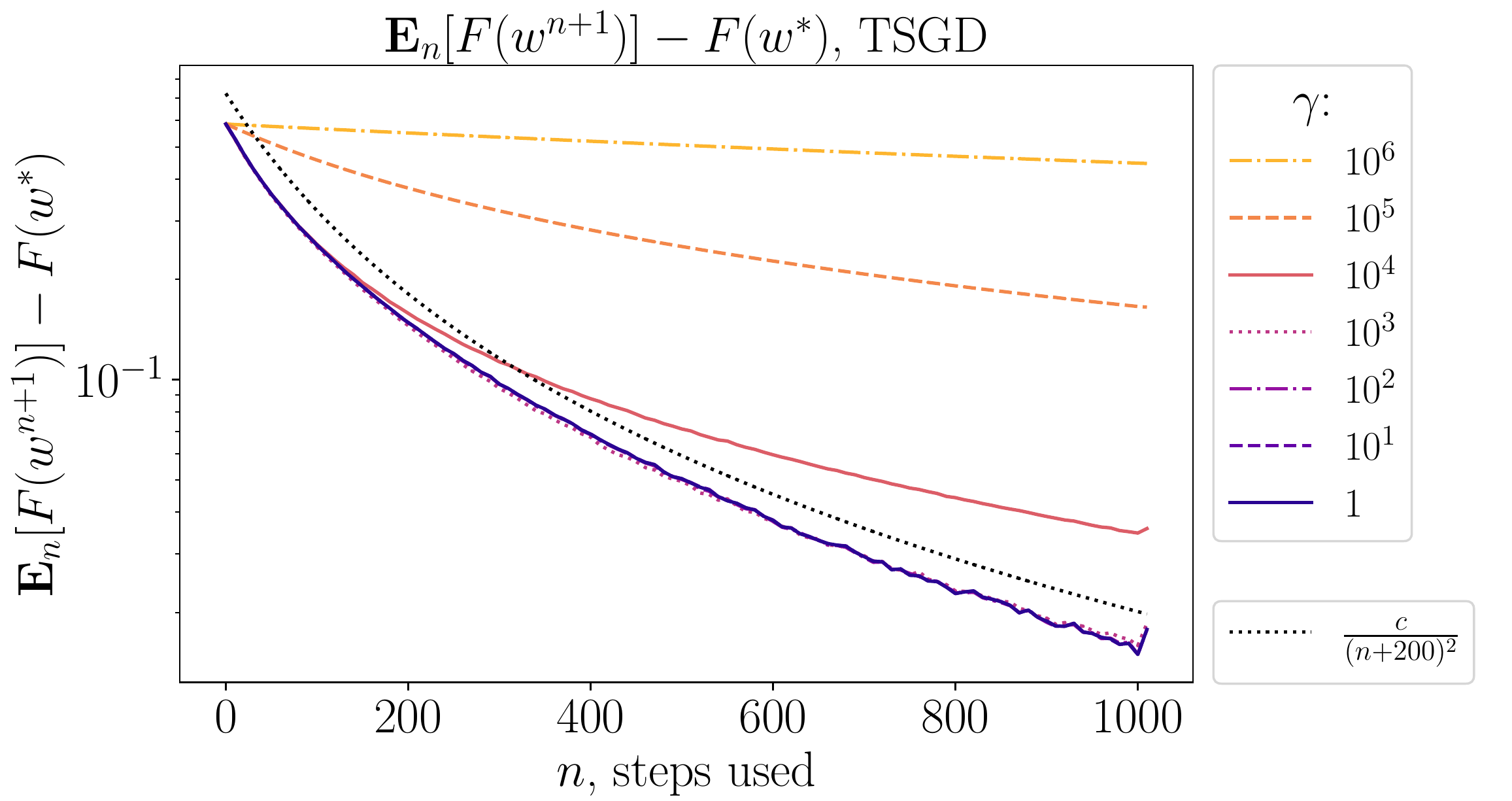}
  \includegraphics[width=0.49\textwidth]{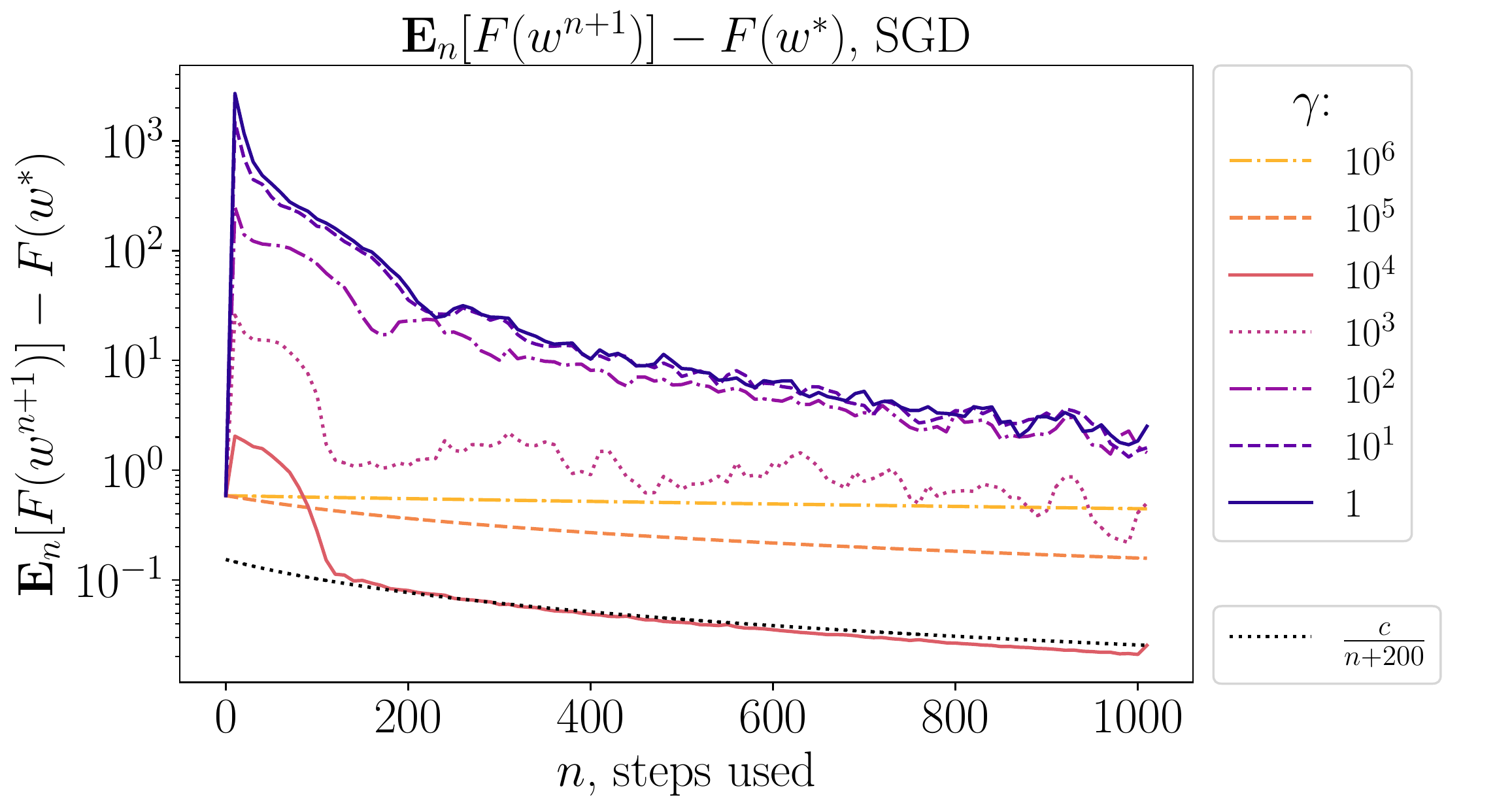}
  \caption{TSGD (left) versus SGD (right), different step sizes, SVM, 
    rcv1.binary}
  \label{fig:lossSVM_cat}
\end{figure}

\subsection{Neural network}

We used a fully connected neural network with one hidden layer containing $100$ 
neurons. 
The activation function was $f(x) = \max\{0,x\}$ on the hidden 
layers and $f(x) = 1 / (1 + \exptext{-x})$ on the output layer. The regularization 
parameter was again $\lambda = 10^{-5}$.
We allowed a maximum of $10$ epochs, and used a batch size which was $1\%$ of the 
amount of samples. We stored every tenth value $F(w^n)$ in order 
to save computational costs.
For the steps size sequence $\{\alpha_n\}_{n \in \N}$ with $\alpha_n = 
\frac{\vartheta}{n + \gamma}$, $n \in \N$, we chose $\vartheta = 10^5 = 
\lambda^{-1}$  and we varied $\gamma = 10^m$ for $m \in \{1, \dots, 7\}$. 

The positive effects of TSGD are showing even more clearly in this example. In 
Figure~\ref{fig:lossNN_mush}, we observe that for growing initial step sizes TSGD 
improves, while SGD becomes worse. 
For the second example in Figure~\ref{fig:lossNN_cat} we 
still observe that TSGD is much more stable than SGD even though the best result is 
achieved with $\gamma = 10^4$ and it becomes worse after. Compared to SGD, the 
speed of convergence is faster.

Altogether, we note that \emph{if} we choose the initial step size optimally, we do achieve 
the optimal rate also for SGD. This is, however, difficult to do in a real large-scale 
application, and the method is very sensitive to this choice. In contrast, TSGD performs 
similarly well for many different parameter choices, and is thus not sensitive at all. 
Further, in these examples, the TSGD decay is usually also faster than the best SGD 
decay. Using a different step size sequence for SGD which decreases faster initially and 
slower later might change this result, but it is unclear how to choose this optimally. 
Providing such an automatically tuned step size sequence is also, in fact, essentially 
what TSGD does.

\begin{figure}
  \includegraphics[width=0.49\textwidth]{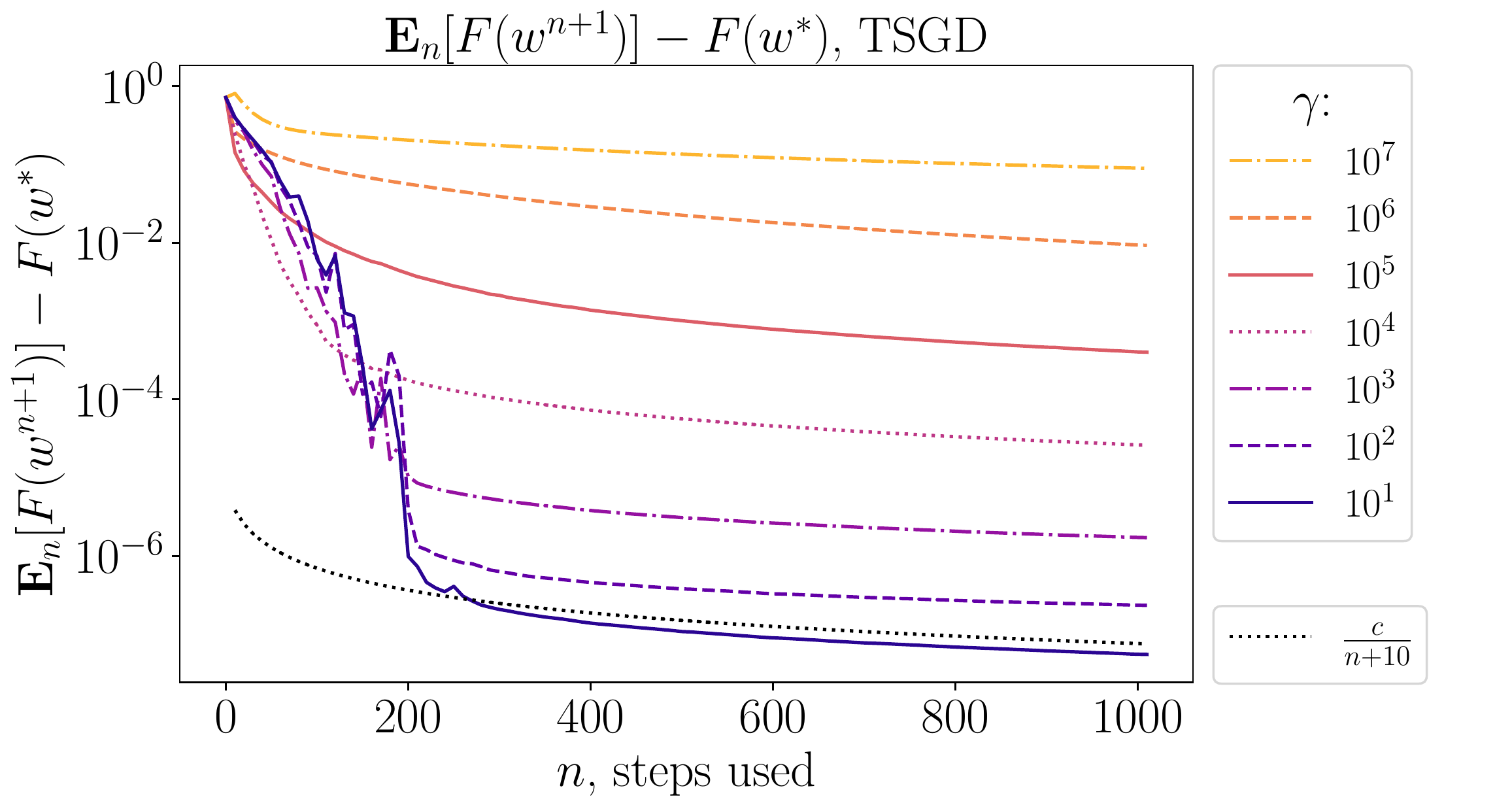}
  \includegraphics[width=0.49\textwidth]{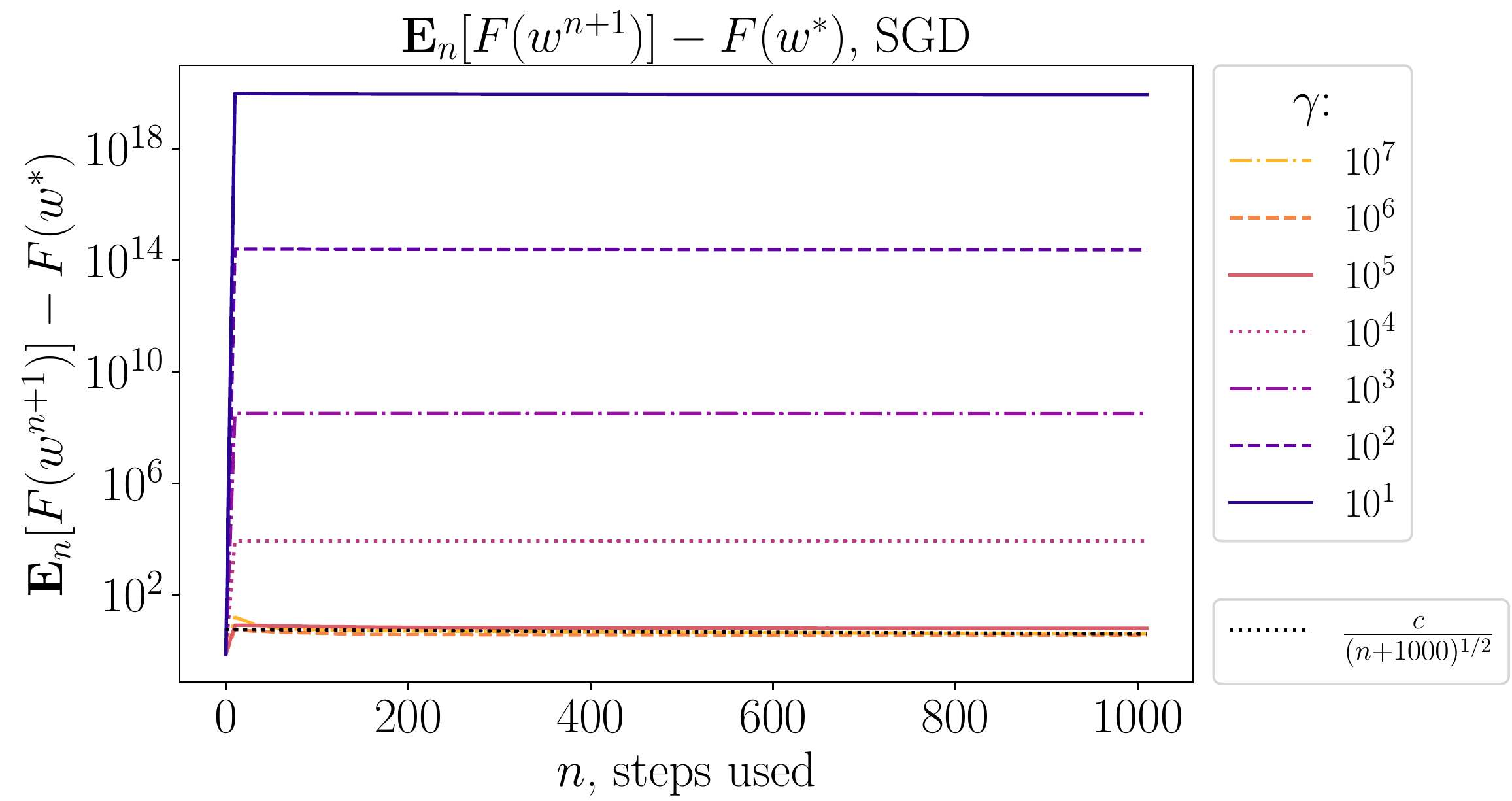}
  \caption{TSGD (left) versus SGD (right), different step sizes, neural 
  network, mushrooms}
  \label{fig:lossNN_mush}
\end{figure}

\begin{figure}
  \includegraphics[width=0.49\textwidth]{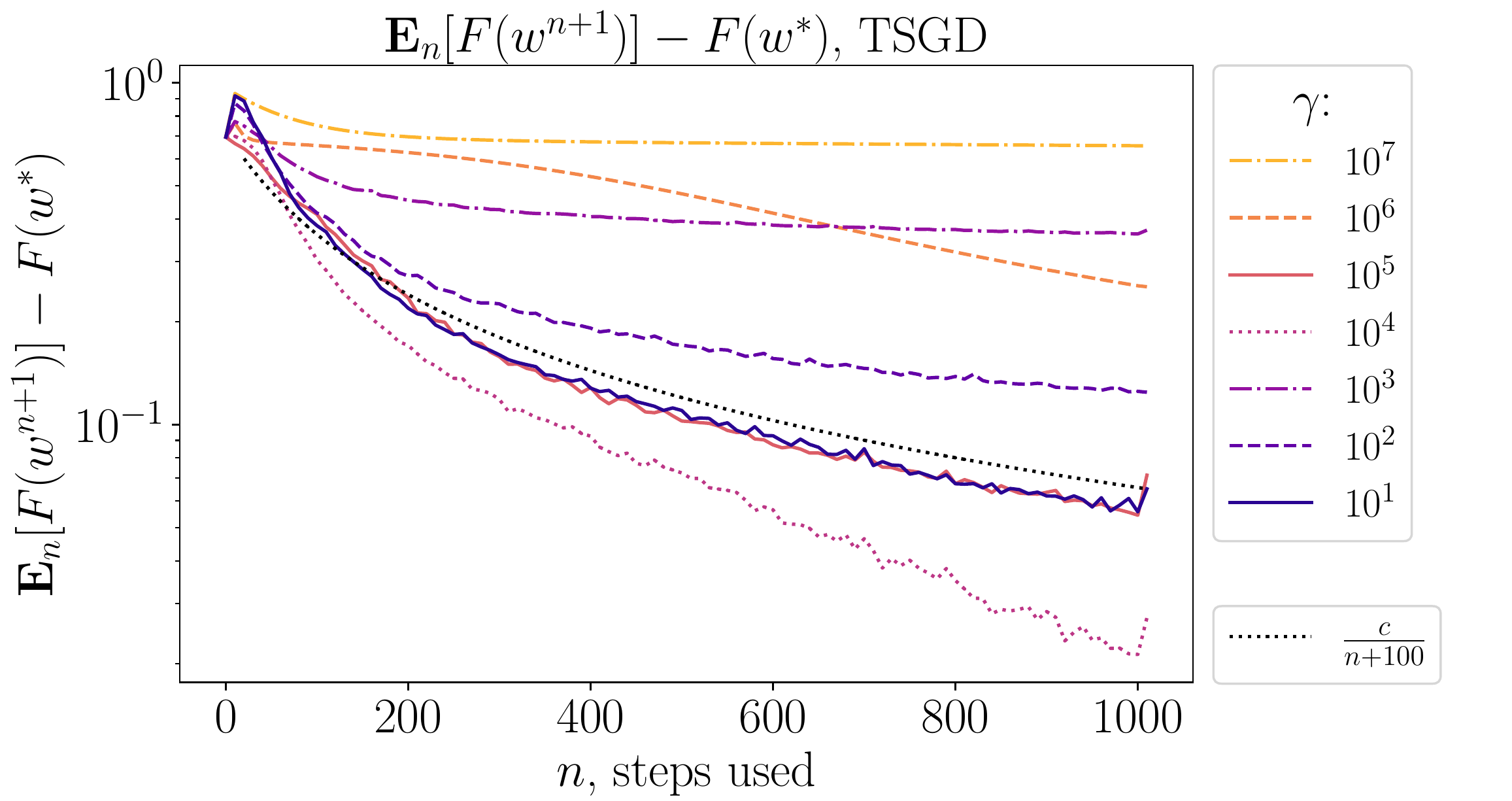}
  \includegraphics[width=0.49\textwidth]{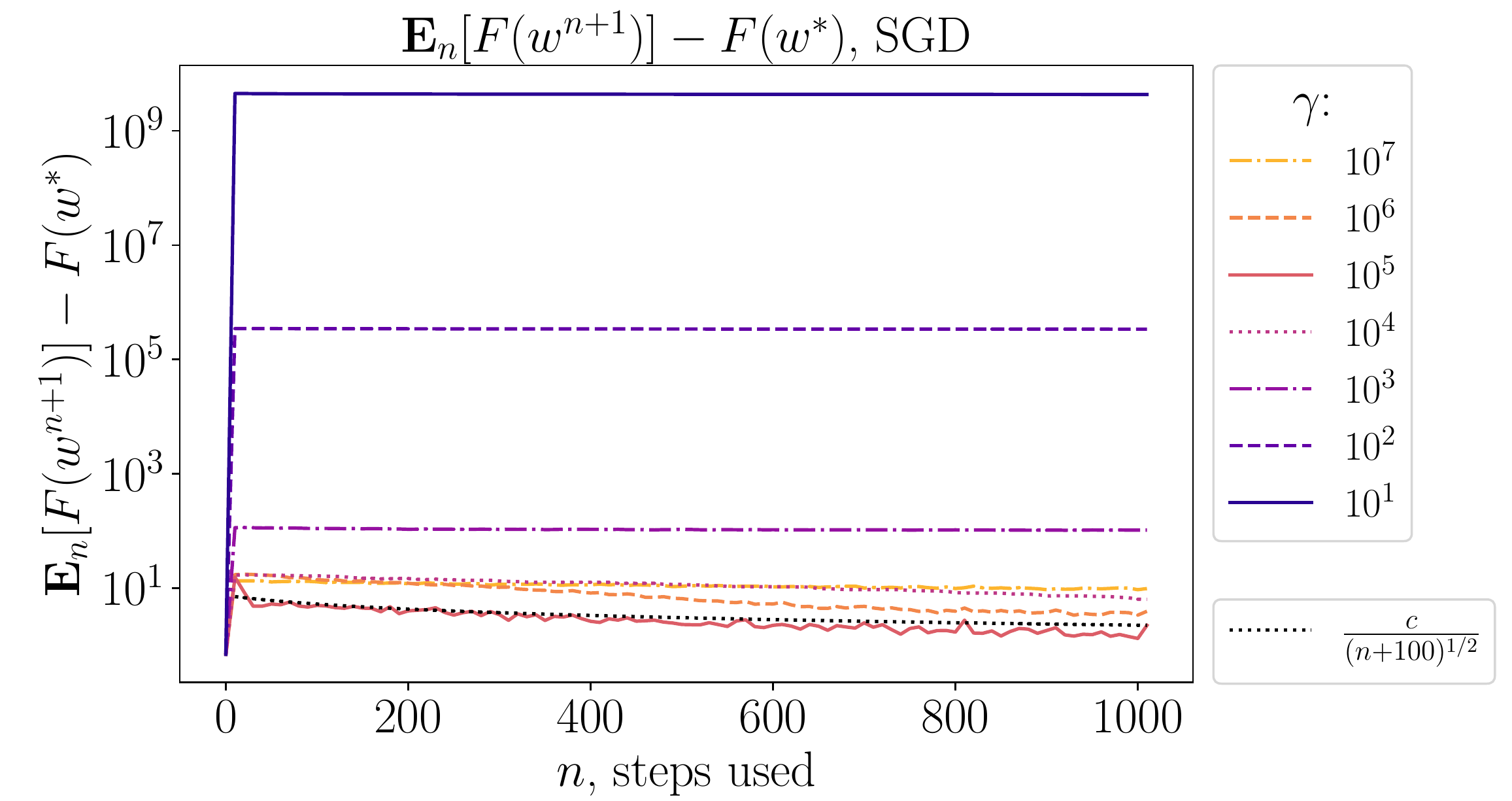}
  \caption{TSGD (left) versus SGD (right), different step sizes, neural 
  network, rcv1.binary}
  \label{fig:lossNN_cat}
\end{figure}

\FloatBarrier

\section{Conclusions}\label{section:conclusions}

We have introduced the TSGD method as an alternative to the well-known SGD method. 
While being comparably inexpensive, TSGD still offers better stability properties in 
comparison to this standard method. 
We have provided a general convergence analysis in an infinite dimensional framework for 
TSGD. While the 
infinite dimensional setting ensures that the error constants are independent of the 
underlying dimension of the problem, our analysis also shows that they are 
only mildly affected by large step sizes. This is in contrast to SGD, where large step 
sizes can lead to extremely large error constants. 
In practice, this means that larger step sizes can be used for TSGD which may lead to 
fast convergence results.
We have also observed that TSGD is much less sensitive to the choice of parameters, in that similar convergence behaviour is often achieved for very different initial step sizes.

The advantages of TSGD were demonstrated in a numerical experiment 
involving a classification problem. We applied both an affine classifier (SVM) and 
a nonlinear classifier (neural network). The affine setting fits into our theory and 
illustrated the theoretical results, while the good performance in the nonlinear framework suggested that there is a wider range 
of applications of the TSGD scheme than those covered by our assumptions.

\appendix
\section{Auxiliary results}\label{section:auxiliary}
This section contains three results that are required for our main theory, but which are 
more generally applicable. The first lemma provides the main algebraic inequalities which 
we base our convergence analysis on:

\begin{lemma} \label{lem:algebraic_inequalities}
  Let $x, y \in (0,\infty)$ and $n,m \in \N$ be given such that $\frac{x}{1 + y} \leq 1$. Then 
  the following inequalities 
  are satisfied:
  \begin{enumerate}[label=(\roman*)] 
    \item $\prod_{i= m}^n \big(1-\frac{x}{i + y}\big) \le 
    \big(\frac{n+1+ y}{m+y}\big)^{-x}$,\\
    \item 
    $\sum_{i=1}^n \frac{1}{(i+y)^2} \prod_{j =i+1}^n \big(1-\frac{x}{j+y} \big)\\
    \leq \exptextb{\frac{x}{1+y}}
    \begin{cases}
      (n+1+y)^{-1}\frac{1}{x-1}, &x \in (1,\infty),\\
      (n+1+y)^{-1} \big(1 + \ln{(n+y)}\big), &x = 1,\\
      (n+1+y)^{-x} \frac{(1+y)^{x-2}(x-2-y)}{x-1},
      &x \in [0,1).
    \end{cases}$
  \end{enumerate}
\end{lemma}

\begin{proof}
  In this proof, we apply the following basic inequalities involving (generalized) 
  harmonic numbers
  \begin{align*}
    \sum_{i=m}^n (i+y)^{-1} &\geq \ln{(n+1+y)} - \ln{(m+y)}, 
    \quad m \in \{1,\dots,n\},\\
    \sum_{i=1}^n{ (i+y)^{p}}
    &\leq 
    \begin{cases}
      \frac{(n+1+y)^{p+1}}{p+1}, &p \in [0,\infty) ,\\
      \frac{(n+y)^{p+1}}{p+1}, &p \in (-1,0),\\
      1 + \ln{(n+y)}, &p = -1,\\
      \frac{(1+y)^{p}(p-y)}{p+1}, &p \in (-\infty,-1),
    \end{cases}
  \end{align*}
  for $y \in (0,\infty)$.
  These inequalities follow by treating the sums as a lower or upper Riemann 
  sums approximating the integral $\int {(u+y)^p \diff{u}}$ over the intervals $[0,n]$, $[1,n]$ 
  or $[0,n+1]$. 
  
  Using the inequality $1 + u \le \exp{u}$ for $u \in [-1, \infty)$, it follows that $0\leq 
  1-\frac{x}{i + y} \le \exptext{-\frac{x}{i+y}}$ is fulfilled for every $i \in \N$ since 
  $\frac{x}{1+y} \leq 1$. It then follows that
  \begin{align*}
    \prod_{i= m}^n \Big(1- \frac{x }{i + y}\Big) 
    &\leq \exptextB{- x \sum_{i=m}^{n} (i + y)^{-1}} \\
    &\leq \exptextb{- x \big(\ln{(n+1+y)} - \ln(m+y)\big)} \\
    &= \exptextB{- x \ln{\Big(\frac{n+1+y}{m+y} \Big)} } 
    = \Big(\frac{n+1+y}{m+y} \Big)^{-x}
  \end{align*}
  from which the first claim follows directly. 
  For the second claim, we use the fact that $\frac{i+1+y}{i+y} = 1 + \frac{1}{i+y} \leq  1 + 
  \frac{1}{1+y}
  \leq \exptext{\frac{1}{1+y}}$ for all $i \in \N$ and find that
  \begin{align*}
    &\sum_{i=1}^n \frac{1}{(i+y)^2} \prod_{j = i+1}^n \Big(1-\frac{x}{j+y} 
    \Big)
    \le \sum_{i=1}^n { \frac{1}{(i+y)^2} \Big(\frac{n+1+y}{i+1+y}\Big)^{-x} } \\
    &\le (n+1+y)^{-x} \sum_{i=1}^n{ \Big(\frac{i+1+y}{i+y}\Big)^{x} (i+y)^{x-2} }\\
    &\le  \exptextB{\frac{x}{1+y}} (n+1+y)^{-x} 
    \sum_{i=1}^n{ (i+y)^{x-2} }\\
    &\leq \exptextB{\frac{x}{1+y}}
    \begin{cases}
      (n+1+y)^{-1}\frac{1}{x-1}, &x \in (1,\infty) ,\\
      (n+1+y)^{-1} \big(1 + \ln{(n+y)}\big), &x = 1,\\
      (n+1+y)^{-x}\frac{(1+y)^{x-2}(x-2-y)}{x-1}, 
      &x \in [0,1),
    \end{cases}
  \end{align*}
  where we applied the basic inequalities from the beginning of the proof.
\end{proof}

\begin{lemma} \label{lem:TaylorEx}
  Given $a,b \in (0,\infty)$, $\frac{-1}{a x + b} \leq - \frac{1}{b} + \frac{a}{b^2} 
  x$ for every $x \in (0,\infty)$. 
\end{lemma}

\begin{proof}
  We consider the function $f \colon [0,\infty) \to \R$ with $f(x) = \frac{-1}{a x + b}$. 
  Then the first and second derivative of $f$ are given by $f'(x) = \frac{a}{(a x + b)^2}$ and 
  $f''(x) = \frac{-2a^2}{(a x + b)^3}$. Using a first-order Taylor expansion of $f$ then shows 
  that
  \begin{align*}
    f(x) = - \frac{1}{b} + \frac{a}{b^2} x - \frac{a^2}{(a \xi + b)^3} x^2
    \leq - \frac{1}{b} + \frac{a}{b^2} x,
  \end{align*}
  where $\xi \in (0,x)$. 
\end{proof}

The final lemma shows that $\nabla F$ does in fact exist and equals $\E_{\xi}\big[ 
\nabla f(\xi, \cdot)\big]$.

\begin{lemma} \label{lem:swap_diff_and_expectation}
  Let Assumption~\ref{ass:fStoch} be fulfilled.
  Then $F = \E_{\xi}[f(\xi, \cdot)]$ is G\^{a}teaux differentiable and its derivative is 
  given by
  \begin{align*}
    \dual{\iota \nabla F (v)}{w} = \E_{\xi}[\dual{\iota \nabla f (\xi, v)}{w}].
  \end{align*}
\end{lemma}
\begin{proof}
  See e.g.~\cite[Lemma 6]{RyuBoyd.2016}.
\end{proof}

\bibliographystyle{siamplain}
\bibliography{ES2021.bib}

\end{document}